\documentclass[11pt,letterpaper]{article}
\usepackage[T1]{fontenc}
\usepackage{amsmath}
\usepackage{amsfonts}
\usepackage{amssymb}
\usepackage{amsthm}
\usepackage{graphicx}
\graphicspath{{img/}}
\usepackage{booktabs}
\usepackage{hyperref}
\usepackage{breakcites}
\usepackage{multicol}
\usepackage{bbm}
\usepackage[ruled,vlined]{algorithm2e}
\usepackage{adjustbox}
\usepackage{multirow}
\usepackage{subfig}
\usepackage{pgfplots}
\usepackage{enumitem}
\usepackage{authblk}
\usepackage{tabularx,array}
\usepackage{placeins}
\usepackage{natbib} 

\newtheorem{theorem}{Theorem}
\newtheorem{remark}{Remark}
\newtheorem{corollary}{Corollary}
\newtheorem{definition}{Definition}
\newtheorem{proposition}{Proposition}

\newtheorem{example}{Example}

\pgfplotsset{compat=1.17}
\usetikzlibrary{calc}

\usepackage{xcolor}

\newcommand{\p}{\mathbb{P}}
\newcommand{\E}{\mathbb{E}}
\newcommand{\eqd}{\overset{\mathrm{d}}{=}}
\newcommand{\id}[1]{\mathbbm{1}_{\{#1\}}}
\newcommand{\bm}[1]{\boldsymbol{#1}}
\newcommand{\esssup}{\operatorname*{ess\,sup}}
\newcommand{\essinf}{\operatorname*{ess\,inf}}

\newcommand{\Unif}{\operatorname{Unif}}
\renewcommand{\d}{\, \mathrm{d}}

\tikzstyle{arrow} = [thick,->,>=stealth]
\usetikzlibrary{trees,shapes,decorations, arrows}
\usetikzlibrary{shapes,positioning,fit}

\usepackage[margin=1.00in]{geometry}

\title{Stochastic representation of Sarmanov copulas}

\author{Christopher Blier-Wong}
\affil{Department of Statistical Sciences, University of Toronto, Canada, \href{mailto:christopher.blierwong@utoronto.ca}{christopher.blierwong@utoronto.ca}}

\date{\today}

\begin{document}

\maketitle

\begin{abstract}
Sarmanov copulas offer a simple and tractable way to build multivariate distributions by perturbing the independence copula.
They admit closed-form expressions for densities and many functionals of interest, making them attractive for practical applications. 
However, the complex conditions on the dependence parameters 
to ensure that Sarmanov copulas are valid limit their application in high dimensions. 
Verifying the $d$-increasing property typically requires satisfying a combinatorial set of inequalities 
that makes direct construction difficult. 
To circumvent this issue, we develop a stochastic representation for bivariate Sarmanov copulas.
We prove that every admissible Sarmanov can be realized as 
a mixture of independent univariate distributions indexed by a latent Bernoulli pair. 
The stochastic representation replaces the problem of verifying copula validity into the problem of ensuring nonnegativity of a Bernoulli probability mass function. 
The representation also recovers classical copula families, 
including Farlie--Gumbel--Morgenstern, Huang--Kotz, and Bairamov--Kotz--Bekçi as special cases. 
We further derive sharp global bounds for Spearman's rho and Kendall's tau.
We then introduce a Bernoulli-mixing construction in higher dimensions, 
leading to a new class of multivariate Sarmanov copulas with 
easily verifiable parameter constraints and scalable simulation algorithms.
Finally, we show that powered versions of bivariate Sarmanov copulas admit 
a similar stochastic representation through block-maximal order statistics.
\end{abstract}

\noindent\textbf{Keywords}: Sarmanov copulas, Farlie--Gumbel--Morgenstern copulas, Huang--Kotz copulas, High-dimensional copulas, Asymmetric copulas, Measures of multivariate association, Stochastic representation

\section{Introduction}

When building multivariate statistical models, it is often more convenient to 
treat dependence separately from the marginal behaviour.
In many cases, parametric or semiparametric models capture marginal behaviour well, 
particularly if the statistician has substantial domain knowledge \cite{mcneil2015quantitative}.
By contrast, dependence is rarely captured by linear correlation, 
can vary substantially across the support, and is sensitive to monotone marginal transformations.
Copulas provide a standard framework for representing this dependence structure seperately from the margins. 
They enable the combination of arbitrary marginals into a multivariate distribution 
and may yield families with interpretable association parameters.
We refer the reader to \cite{nelsen2006introduction, joe2015dependence, durante2015principles} 
for comprehensive introductions to copula theory and applications, 
and \cite{grosser2021copulae, genest2024copula} for recent reviews. 

Copulas also enable a modular approach to dependence modelling, allowing marginal distributions to be specified or 
estimated separately using the most appropriate tools for each margin. 
Once the margins are fixed, the statistician can focus on selecting a copula family that captures 
relevant dependence properties such as tail behaviour, asymmetry, or exchangeability. 
This flexibility has made copulas a standard tool in insurance and finance, where rigorous risk management and 
scenario analysis require careful multivariate modelling \cite{cherubini2004copula, denuit2006actuarial, mcneil2015quantitative}.

A $d$-copula is a distribution function on $[0,1]^d$ with uniform univariate margins. 
Formally, a function $C:[0,1]^d\to[0,1]$ is a copula if it satisfies three conditions: 
\begin{enumerate}[label=(\roman*)]
   \item groundedness, meaning $C(\bm u)=0$ whenever at least one component of $\bm u$ equals $0$; 
   \item uniform margins, so that $C(1,\dots,1,u_j,1,\dots,1)=u_j$ for each $j$; and 
   \item $d$-increasing, meaning that all rectangular increments of $C$ defined by coordinatewise intervals are nonnegative.
\end{enumerate}
An explicit formulation of the latter condition is that for all vectors $\bm a\le \bm b$ in $[0,1]^d$,
\begin{equation}\label{eq:d-increasing}
   V_C([\bm a,\bm b])=\sum_{i_1=0}^1\dots\sum_{i_d=0}^1 (-1)^{i_1+\dots+i_d} C(u_{1i_1},\dots,u_{di_d}) \ge 0,
\end{equation}
where $u_{j0}=a_j$ and $u_{j1}=b_j$.

Copulas allows one to separate the dependence from the marginal distributions through Sklar's theorem \cite{sklar1959fonctions}. 
If $H$ is a $d$-variate distribution function with univariate margins $F_1,\dots,F_d$, then there exists a copula $C$ such that
$$H(x_1,\dots,x_d)=C\left(F_1(x_1),\dots,F_d(x_d)\right), \qquad (x_1,\dots,x_d)\in\mathbb R^d.$$
If the margins are continuous, this copula is unique. 
Conversely, for any copula $C$ and any collection of univariate margins $F_1,\dots,F_d$,
the right-hand side defines a valid multivariate distribution with those margins. 

The demand for high-dimensional dependence models is driven by the scale of modern data 
and the fact that dependence can have a substantial impact on risk assessment and decision-making. 
In finance and insurance, portfolios often involve hundreds or thousands of policies, 
creating high-dimensional loss distributions where tail events arise from the interaction of many risk factors \cite{mcneil2015quantitative}. 
Similarly, in environmental sciences such as climate, hydrology, and air quality, 
multivariate dependence across space and time affects the risk of joint extremes and compound events \cite{ribeiro2020risk, naseri2022bayesian}.

Modelling high-dimensional dependence structures usually involves a trade-off between flexibility and tractability. 
Vine copulas, for instance, offer great flexibility but are difficult to scale. 
Conversely, simpler parsimonious families have simpler parameter constraints but often too restrictive, 
e.g., Archimedean copulas are limited to exchangeable dependence structures. 
There is therefore a need for alternative models that balance flexibility, interpretability, and scalability.
Sarmanov copulas are an attractive option because they apply perturbations to the independence copula and admit closed-form expressions for many quantities of interest,
but their admissibility conditions become increasingly difficult to satisfy as dimensions grow, and we are not aware of 
any existing multivariate Sarmanov families other than the simplest forms. 
To overcome this, we introduce a stochastic representation for bivariate Sarmanov copulas. 
Building on this, we then propose a new class of multivariate Sarmanov copulas by extending the bivariate stochastic representation to higher dimensions, and we develop their properties.
Parameter constraints under the stochastic representation are easy to verify, and simulation is straightforward and scalable.

Sarmanov copulas are widely used in risk management because they provide flexible dependence structures 
without sacrificing analytical tractability, particularly when combined with phase-type or mixed Erlang distributions. 
See, e.g., \cite{hashorva2015sarmanov}, where the authors derive aggregation formulas and capital allocation rules 
when risks follow Sarmanov distributions. 
They show that the resulting aggregate random variable remains within the class of mixed Erlang distributions. 
In \cite{ratovomirija2017multivariate}, the authors used Sarmanov-mixed Erlang risks 
and studied the properties of the aggregate risk and capital allocation.
They also studied how the choice of kernel affects the range of attainable dependence. 
More recently, \cite{vernic2022sarmanov} modelled the dependence between claim frequency and severity.
They showed that ignoring the dependence significantly biases insurance premiums. 
To date, however, most research has focused on bivariate cases where admissibility constraints are easier to handle.

Some of the earliest and most widely used copula families are Farlie--Gumbel--Morgenstern (FGM) copulas, 
proposed in \cite{eyraud1936principes,gumbel1960bivariate,morgenstern1956einfache}. 
Subsequently, \cite{farlie1960performance} introduced a general class of bivariate distributions 
that generalize the FGM construction, whose shape is given by $C(u_1,u_2)=u_1u_2(1+\theta A(u_1)B(u_2))$ 
for kernel functions $A$ and $B$ satisfying certain regularity conditions, 
and a scalar parameter $\theta$ chosen to ensure that $C$ is a copula. 
In \cite{sarmanov1966generalized}, the author introduced a slightly more general formulation of the Farlie copula, 
whose densities are of the form
$$c(u_1, u_2) = 1 + a \phi_1(u_1) \phi_2(u_2), \qquad (u_1,u_2)\in[0,1]^2,$$
where $\int_0^1 \phi_j(u)\d u=0$ for $j=1,2$, and $a$ is a scalar parameter chosen to ensure nonnegativity of $c$, 
leading to the same class of copulas as Farlie copulas under common regularity conditions. 
We refer to this family as Sarmanov copulas; which are of the form
\begin{equation}\label{eq:farlie-2d}
   C(u_1,u_2)=u_1u_2 + a g_1(u_1)g_2(u_2), \qquad (u_1,u_2)\in[0,1]^2,
\end{equation}
for kernel functions $g_1$ and $g_2$ such that $g_1(0)=g_1(1)=0$ and $g_2(0)=g_2(1)=0$, 
and a scalar parameter $a$ chosen to satisfy regularity conditions ensuring that $C$ is a copula. 
Such models are attractive because they admit tractable expressions for densities and dependence measures, 
and they may be interpreted as perturbations of independence.
In this class of copulas, the kernels $g_1$ and $g_2$ control the shape of the dependence, 
while the parameter $a$ controls its strength.

However, ensuring that \eqref{eq:farlie-2d} is a valid copula requires 
verifying the $2$-increasing property in \eqref{eq:d-increasing}. 
In practice, finding the admissible range of $a$ for given kernels $g_1$ and $g_2$ is challenging, 
and most research proposing new Sarmanov families devote considerable effort to this task. 
Sarmanov extensions are tractable once admissibility is established, 
but verifying admissibility directly is typically the most challenging step.

The classical Farlie--Gumbel--Morgenstern copula corresponds to the choice $g_1(u)=u(1-u)$ and $g_2(u_2)=u_2(1-u_2)$, leading to 
$$C(u_1,u_2)=u_1u_2\left(1+\theta(1-u_1)(1-u_2)\right), \qquad (u_1,u_2)\in[0,1]^2,$$
with admissible parameter $\theta\in[-1,1]$. 
This family is analytically convenient and admits simple expressions for rank correlations. 
For instance, in the bivariate case, both Kendall's $\tau$ and Spearman's $\rho$ are linear functions of $\theta$, 
although the maximum attainable dependence is limited to $|\rho|\le 1/3$ and $|\tau|\le 2/9$. 
Introductory texts often use the FGM family to illustrate copula theory, 
as its simple structure and closed-form properties make it easy to analyze, 
see, for instance, \cite{genest2007everything}. 

In dimension $d$, the multivariate FGM copula takes the polynomial form
$$C(\bm u)=\left(\prod_{m=1}^d u_m\right)\left(1+\sum_{k=2}^d \sum_{1\le j_1<\cdots<j_k\le d}\theta_{j_1\cdots j_k} \bar u_{j_1}\cdots \bar u_{j_k}\right),\qquad \bm u \in [0,1]^d,$$
where $\bar u_j=1-u_j$ and the parameter vector $\bm\theta$ has $2^d-d-1$ components. 
The admissible set of parameters is nontrivial: these parameters must satisfy the $2^d$ linear inequalities 
\begin{equation}\label{eq:fgm-admissibility}
   \sum_{k=2}^d \sum_{1\le j_1<\cdots<j_k\le d} \varepsilon_{j_1}\cdots \varepsilon_{j_k} \theta_{j_1\cdots j_k} \ge -1.
\end{equation}
for all sign vectors $\bm\varepsilon\in\{-1,1\}^d$.

In \cite{huang1999modifications}, the authors propose polynomial extensions of the FGM family 
designed to increase the achievable level of positive dependence while retaining closed-form tractability. 
They consider the family \eqref{eq:HKI} in Section 2 and the family \eqref{eq:HKII} in Section 3:
\begin{equation}\label{eq:HKI}
   C(u_1,u_2)=u_1u_2\left(1+a(1-u_1^p)(1-u_2^p)\right), \quad p> 1;
\end{equation}
\begin{equation}\label{eq:HKII}
   C(u_1,u_2)=u_1u_2\left(1+a(1-u_1)^q(1-u_2)^q\right), \quad q> 0.
\end{equation}
We will refer to these as the Huang--Kotz I (HKI) and Huang--Kotz II (HKII) families, respectively. 
These extensions are both Sarmanov copulas with symmetric kernels 
$g_1(u)=g_2(u)=u(1-u^p)$ and $g_1(u)=g_2(u)=u(1-u)^q$, respectively. 
Their analysis shows that within such polynomial extensions, 
the maximal attainable positive correlation can be slightly increased beyond the classical FGM bounds, 
but the negative dependence remains limited to $-1/3$ and $-2/9$ for Spearman's $\rho$ and Kendall's $\tau$, respectively. 
The admissible parameter regions for $a$ in \eqref{eq:HKI} and \eqref{eq:HKII} are derived by direct verification of 2-increasing. 
Most of \cite{huang1999modifications} is devoted to this task.
\footnote{The title of their paper is ``Modifications of the Farlie--Gumbel--Morgenstern distributions. A tough hill to climb,'' 
reflecting the authors' resolve in deriving these admissibility constraints.}

Verifying admissibility becomes increasingly difficult in higher dimensions. 
A $d$-copula must be $d$-increasing, a condition that is already combinatorial in its most explicit form. 
For multivariate FGM copulas, the parameter dimension grows as $2^d-d-1$, and 
the admissible region is defined by the $2^d$ linear constraints in \eqref{eq:fgm-admissibility}. 
For general Sarmanov families, particularly those with asymmetric kernels or multiple parameters, 
the admissibility problem is particularly challenging, and we have not found any
literature on constructing or analyzing multivariate Sarmanov copulas beyond simple extensions of the FGM copula.

The authors of \cite{blier-wong2022stochastic} establish a one-to-one correspondence between multivariate FGM copulas 
and multivariate Bernoulli distributions with symmetric margins. 
The dependence parameters coincide with mixed central moments of a latent Bernoulli vector. 
Specifically, if $\bm I=(I_1,\dots,I_d)$ satisfies $\p(I_j=1)=1/2$, then
$$C(\bm u)=\E\left[\prod_{m = 1}^d u_m \left(1 + (-1)^{I_m} \bar u_m\right)\right],$$
is an FGM copula with parameters given by
$$\theta_{j_1\cdots j_k} = (-2)^k \E\left[\prod_{n=1}^k\left(I_{j_n} - \frac12\right)\right].$$
Further, let $\bm{U}_{[j]}=(U_{1,[j]},\dots,U_{d,[j]})$ have i.i.d. components with $U_{k,[j]}\sim \mathrm{Beta}(2-j,1+j)$ 
for $j\in\{0, 1\}$, and define $\bm U \eqd (1-\bm I)\bm{U}_{[0]} + \bm I \bm{U}_{[1]}$. 
Then, $\bm U$ has uniform margins and FGM copula with parameters as above. 
The FGM copula is therefore the copula of a latent Bernoulli mixture of independent Beta components. 
This representation characterizes admissibility of $\bm\theta$ as being equivalent to nonnegativity of a multivariate Bernoulli probability mass function on $\{0,1\}^d$,
$$f_{\bm I}(\bm i) = 2^{-d}\left(1+\sum_{k=2}^d \sum_{j_1<\cdots<j_k}\theta_{j_1\cdots j_k}(-1)^{i_{j_1}+\cdots+i_{j_k}}\right)\ge 0,\qquad \bm i\in\{0,1\}^d,$$ 
and allows for a natural simulation procedure based on sampling from the Bernoulli vector $\bm I$ and the Beta marginals.

This representation transforms the verification of $d$-increasing into a finite set of nonnegativity constraints for a Bernoulli pmf. 
Any valid symmetric Bernoulli distribution for $\bm I$ immediately leads to admissible parameters $\bm\theta$
The multivariate Bernoulli pmf constraints are well understood \cite{sharakhmetov2002characterization, fontana2018representation} 
and various structured subclasses (exchangeable, Markov, factor models) allow scalable inference and simulation. 
This Bernoulli-driven viewpoint also facilitates stochastic ordering analysis. Subsequent work has exploited this relationship in exchangeable settings \cite{blier-wong2024exchangeable}.

In \cite{blier-wong2024new}, the authors propose to generalize the Bernoulli-mixture representation of FGM copulas. 
Starting with the existing stochastic representation of FGM copulas, they modify the marginal distributions from Beta 
to other families while retaining the Bernoulli-mixture structure, and the Huang--Kotz I family. 
In particular, let $\bm I=(I_1,\dots,I_d)$ be a $d$-variate Bernoulli vector with $\p(I_m=1)=\pi_m\in (0,1)$ 
(the difference from the FGM case is that the margins need not be symmetric). 
Let $\bm U_0=(U_{0,1},\dots,U_{0,d})$ and $\bm U_1=(U_{1,1},\dots,U_{1,d})$ be vectors of independent $\mathrm{Unif}(0,1)$ random variables, independent of $\bm I$. Define
$$ \bm U  \eqd \bm{U}_0^{\bm 1-\bm \pi} \bm U_1^{\bm I} =\left(U_{0,1}^{1-\pi_1}U_{1,1}^{I_1},\dots, U_{0,d}^{1-\pi_d}U_{1,d}^{I_d}\right),$$
with all powers and products taken componentwise. 
This construction leads to standard uniform margins for $\bm U$. 
For $d=2$, the cumulative distribution function (cdf) of $\bm U$ is given by 
$$C(u_1,u_2) = u_1u_2\left(1+\theta_{12}\left(1-u_1^{\frac{\pi_1}{1-\pi_1}}\right)\left(1-u_2^{\frac{\pi_2}{1-\pi_2}}\right)\right), \qquad (u_1,u_2)\in[0,1]^2,$$  
with
$$\theta_{12}=\E\left[\frac{I_1-\pi_1}{\pi_1}\frac{I_2-\pi_2}{\pi_2}\right].$$
In the symmetric case $\pi_1=\pi_2$ coincides with the Huang--Kotz I family of copulas in \eqref{eq:HKI}. 
See also \cite{cossette2025generalized} for a higher-dimensional analysis and geometric interpretation of this construction. 

Since FGM and Huang--Kotz I families of copulas both admit such a Bernoulli-mixture representation, 
a natural question is whether this construction can be extended to more general Sarmanov copulas, 
including asymmetric kernels and other Huang--Kotz variants. 
The main result of this paper is to show that this is indeed the case.
We develop a unifying stochastic representation for Sarmanov copulas 
of the form $C(u_1,u_2)=u_1u_2+a g_1(u_1)g_2(u_2)$. 
The main theorem shows that, under mild regularity on the kernels $g_1$ and $g_2$, 
any bivariate Sarmanov copula can be realized as the cdf of a 
Bernoulli-mixture between two marginal distributions, 
so that admissibility reduces to finitely many nonnegativity constraints 
for a bivariate Bernoulli pmf rather than a large set of inequalities. 
This leads to a constructive method for building and simulating from asymmetric Sarmanov copulas and 
provides a common probabilistic mechanism that recovers classical families (including FGM and Huang--Kotz-type perturbations) as special cases. 
We then introduce a Bernoulli-mixing construction in higher dimensions, yielding a copula with a generalized FGM-type subset expansion 
that enables scalable simulation and interpretable dependence modelling, and we develop extremal dependence in key subclasses. 
Finally, we broaden the framework to powered perturbations $C(u_1,u_2)=u_1u_2(1+ah_1(u_1)h_2(u_2))^r$, for $r>0$, 
giving new tractable families and corresponding dependence bounds and representations. 

The Sarmanov family of copulas is closely related to other constructions that have appeared under different names in the literature. 
The Sarmanov copulas studied here are very close to the Farlie class of copulas: 
Farlie \cite{farlie1960performance} considered copulas of the form $C(u_1,u_2)=u_1u_2(1+ah_1(u_1)h_2(u_2))$. 
Farlie copulas are a special case of Sarmanov copulas, covering almost all Sarmanov copulas under common regularity conditions 
with $u h_1(u)$ and $u h_2(u)$ absolutely continuous and having bounded derivatives. 
The Sarmanov construction was also rediscovered by \cite{rodriguez-lallena2004new}. 
Under common regularity conditions, their definitions coincide after a change of parametrization. 
For clarity, we refer to the family throughout as ``Sarmanov'' copulas, even if the Farlie representation predates Sarmanov's work. 
In Appendix \ref{app:comparison-three-classes}, we provide a detailed comparison of the three formulations, 
and explain the precise conditions under which they are equivalent. Note that the main results of this paper apply to all three families, 
hence providing a unifying stochastic representation across these closely related constructions.

The remainder of this paper is organized as follows. 
In Section \ref{sec:farlie}, we review Sarmanov copulas and present the main theorem establishing a Bernoulli-mixture representation, 
illustrate the stochastic representation of existing copulas in the literature, 
and find the attainable association coefficients within the bivariate Sarmanov class. 
In Section \ref{sec:dvariate}, we introduce a Bernoulli-mixing construction in higher dimensions, 
creating a new class of multivariate Sarmanov copulas that has an FGM-type subset expansion. 
In Section \ref{sec:power}, we extend the construction to powered bivariate Sarmanov copulas. 
Section \ref{sec:conclusion} summarizes the contributions and discusses extensions. 
Appendix \ref{sec:proofs} contains the proofs of the main results; 
Appendix \ref{sec:examples} provides additional examples of Sarmanov copulas and their stochastic representations, 
while Appendix \ref{app:comparison-three-classes} compares Sarmanov copulas to related constructions in the literature.

\section{Sarmanov copulas and a stochastic representation}\label{sec:farlie}

In this section, we develop a stochastic representation for general bivariate Sarmanov copulas with Bernoulli distributions. We will show that Sarmanov copulas can be represented as the copula of a latent Bernoulli mixture of independent random variables with suitable marginal distributions. This construction generalizes the Bernoulli-mixture representations of FGM and Huang--Kotz I copulas discussed above.

\subsection{Sarmanov copulas}\label{ss:farlie-def}

Consider the classical Sarmanov perturbation of the product copula,
$$C(u_1, u_2) = u_1 u_2 + a g_1(u_1) g_2(u_2), \qquad (u_1, u_2) \in [0,1]^2,$$
where the kernels $g_1, g_2:[0,1]\to\mathbb R$ may be different (hence asymmetric). Even in dimension 2, determining the admissible values of $a$ that ensure that $C$ is a copula is typically achieved by verifying the 2-increasing property. Under absolute continuity, this reduces to verifying nonnegativity of a mixed derivative on $[0,1]^2$, but doing so directly can still be tedious because the admissible interval depends on the detailed range of certain kernel derivatives.

We assume $g_1$ and $g_2$ are absolutely continuous on $[0,1]$, satisfy $g_1(0)=g_1(1)=0$ and $g_2(0)=g_2(1)=0$, and have a.e. derivatives $\phi_1,\phi_2\in L^\infty([0,1])$. Define $\phi_1(u)=g_1'(u)$ and $\phi_2(v)=g_2'(v)$ a.e., and define the constants
$$\Lambda_1=\frac{1}{\esssup_{u\in[0,1]}\phi_1(u)}\in(0,\infty),\qquad \lambda_1=\frac{1}{\essinf_{u\in[0,1]}\phi_1(u)}\in(-\infty,0),$$
$$\Lambda_2=\frac{1}{\esssup_{v\in[0,1]}\phi_2(v)}\in(0,\infty),\qquad \lambda_2=\frac{1}{\essinf_{v\in[0,1]}\phi_2(v)}\in(-\infty,0).$$
We have that $\lambda_1<0<\Lambda_1$ since absolute continuity implies
$0=g_1(1)-g_1(0)=\int_0^1\phi_1(u)\d u$, so $\phi_1$ must take both signs unless it is identically zero (in which case $C$ reduces to the independence copula). The same reasoning applies to $\phi_2$. Under these assumptions, the mixed derivative of a Sarmanov copula is
$$\frac{\partial^2}{\partial u_1 \partial u_2}C(u_1,u_2)=1+a\phi_1(u_1)\phi_2(u_2),$$
so (in the absolutely continuous setting) $C$ is a copula if and only if $1+a\phi_1(u_1)\phi_2(u_2)\ge 0$ a.e. This yields an explicit admissible interval for $a$. Still, the point of the main theorem below is that we do not need to analyze these inequalities directly: the stochastic representation automatically ensures admissibility.

In what follows, we will use a bivariate Bernoulli vector $(I_1,I_2)$ with success probabilities $\p(I_1=1)=\pi_1$ and $\p(I_2=1)=\pi_2.$ Its dependence structure will be specified through a single parameter given by the centred moment
\begin{equation}\label{eq:theta-param}
   \theta=\E\left[\frac{(I_1-\pi_1)(I_2-\pi_2)}{\pi_1\pi_2}\right]=\frac{\operatorname{Cov}(I_1,I_2)}{\pi_1\pi_2}.
\end{equation}
It follows that the probability mass function of $(I_1,I_2)$ is given by
\begin{equation}\label{eq:bernoulli-pmf}
   \p(I_1 = i_1, I_2 = i_2) = 
\begin{cases} 
   \pi_1 \pi_2 (1 + \theta), & (i_1,i_2) = (1,1) \\ 
   \pi_1 (1 - \pi_2) - \pi_1 \pi_2 \theta, & (i_1,i_2) = (1,0) \\ 
   (1 - \pi_1) \pi_2 - \pi_1 \pi_2 \theta, & (i_1,i_2) = (0,1) \\ 
   (1 - \pi_1)(1 - \pi_2) + \pi_1 \pi_2 \theta, & (i_1,i_2) = (0,0) 
\end{cases}.
\end{equation}

\subsection{Main result: stochastic representation for Sarmanov copulas}

We are now ready to present the main theorem, which establishes a Bernoulli-mixture representation for general bivariate Sarmanov copulas under the regularity conditions above. The proof is given in Appendix \ref{sec:proofs}.
\begin{theorem}\label{thm:stochastic}
   Assume $g_1$ and $g_2$ satisfy the absolute-continuity and bounded-derivative conditions stated in Section \ref{ss:farlie-def}. Define scaling constants $\pi_1=\Lambda_1/(\Lambda_1 - \lambda_1)\in(0,1)$ and $\pi_2=\Lambda_2/(\Lambda_2-\lambda_2)\in(0,1).$ Define two auxiliary cdfs on $[0,1]$ for each margin:
$$F_{U,[0]}(u)=u-\Lambda_1 g_1(u),\qquad F_{U,[1]}(u)=u-\lambda_1 g_1(u),$$
$$F_{V,[0]}(v)=v-\Lambda_2 g_2(v),\qquad F_{V,[1]}(v)=v-\lambda_2 g_2(v).$$
Let $(U_{[0]},U_{[1]},V_{[0]},V_{[1]})$ be mutually independent, with $U_{[i]}\sim F_{U,[i]}$ and $V_{[i]}\sim F_{V,[i]}$ for $i\in\{0,1\}$. Let $(I_1,I_2)$ be a bivariate Bernoulli vector independent of the above with the parametrization in \eqref{eq:bernoulli-pmf}. Consider the random vector $(U_1,U_2)$ defined by
$$U_1=(1-I_1)U_{[0]}+I_1U_{[1]},\qquad U_2=(1-I_2)V_{[0]}+I_2V_{[1]}.$$
Then, 
\begin{enumerate}
   \item $U_1\sim\mathrm{Unif}(0,1)$ and $U_2\sim\mathrm{Unif}(0,1)$.
   \item The copula of $(U_1,U_2)$ satisfies \eqref{eq:farlie-2d} with $a=\Lambda_1\Lambda_2\theta$, where $\theta$ is given by \eqref{eq:theta-param}.
\end{enumerate}
\end{theorem}

A direct consequence of Theorem \ref{thm:stochastic} is that the admissible range of $a$ ensuring that \eqref{eq:farlie-2d} is a valid copula is determined by the existence of a bivariate Bernoulli distribution with parameter $\theta=a/(\Lambda_1\Lambda_2)$. To build a valid asymmetric Sarmanov copula $C(u_1,u_2)=u_1u_2+a g_1(u_1) g_2(u_2)$, one can (i) compute $\Lambda_1,$ $\Lambda_2,$ $\lambda_1,$ $\lambda_2$, (ii) compute $\pi_1,\pi_2$ and the auxiliary cdfs $F_{U,[0]},F_{U,[1]},F_{V,[0]},F_{V,[1]}$, and (iii) choose $a$ within the admissible range so that $\theta=a/(\Lambda_1\Lambda_2)$ yields a valid bivariate Bernoulli pmf through \eqref{eq:bernoulli-pmf}. Validating the 2-increasing property is thus replaced by checking the nonnegativity of four Bernoulli pmf entries; the feasibility of $a$ is equivalent to the existence of a bivariate Bernoulli distribution with parameter $\theta=a/(\Lambda_1\Lambda_2)$. The Bernoulli pmf constraints are well understood; see, for instance, \cite{fontana2018representation}. This approach provides a constructive method for building and simulating asymmetric Sarmanov copulas. Equivalently, we have directly that
$$a\in\left[ -\min\left(\Lambda_1\Lambda_2, \lambda_1\lambda_2\right), \min\left(\Lambda_1|\lambda_2|, \Lambda_2|\lambda_1|\right)\right].$$

Further, if $g_1=g_2=g$, then $\phi_1=\phi_2=\phi$, and one recovers the symmetric Sarmanov class $C(u_1, u_2)=u_1u_2 + a g(u_1) g(u_2)$. The admissible range reduces to
$$a\in\left[-\min\left(\lambda^2, \Lambda^2 \right),\Lambda|\lambda|\right],$$
where $\Lambda^{-1}=\sup\phi$ and $\lambda^{-1}=\inf\phi$, and the same Bernoulli mixture construction applies. 

Theorem \ref{thm:stochastic} also provides a geometric interpretation of the Sarmanov class. For fixed kernels $g_1,g_2$ (implying fixed $\pi_1,\pi_2,\Lambda_1,\Lambda_2$), admissibility reduces to choosing a distribution for the bivariate Bernoulli random vector $(I_1,I_2)$ with prescribed margins. The copula is the image of the distribution of $(I_1, I_2)$ under an affine map defined by the stochastic representation. Since the Fréchet class of bivariate Bernoulli random vectors with given margins is a line segment with exactly two extreme points, the corresponding family of Sarmanov copulas is itself a line segment in the convex set of copulas: every admissible copula can be written as a convex combination of the two extremal copulas $C^{-}$ and $C^{+}$ obtained from the extremal Bernoulli couplings, with a unique mixing weight that is an affine function of $\theta=a/(\Lambda_1\Lambda_2)$ (equivalently of $a$).

We now illustrate Theorem \ref{thm:stochastic} through several examples, recovering known stochastic representations for classical Sarmanov copulas such as the FGM, the Huang--Kotz copulas \cite{huang1999modifications}, and the Bairamov--Kotz--Bekçi copulas \cite{bairamov2001new}.

\begin{example}[FGM]
Take $g_1(u)=g_2(u)=u(1-u)$ so that $C(u_1,u_2)=u_1u_2(1+a(1-u_1)(1-u_2))$, which is the bivariate FGM copula. We have $\phi(u)=g'(u)=1-2u,$ so $\Lambda= 1$ and $\lambda=-1$. Hence $\pi=1/2$, and the auxiliary cdfs become
$$F_{U,[0]}(u)=u-u(1-u)=u^2,\qquad F_{U,[1]}(u)=u+u(1-u)=2u-u^2.$$
This is exactly the stochastic representation proposed in \cite{blier-wong2022stochastic}, and the admissible range becomes $-1\le a\le 1$, matching the classical FGM admissible-range statement. We also have $U_{[0]}\sim \mathrm{Beta}(2,1)$ and $U_{[1]}\sim \mathrm{Beta}(1,2)$.
\end{example}

\begin{example}[Huang--Kotz I]
Take $g_1(u)=g_2(u) = u(1-u^p)$, so that $C(u_1,u_2)=u_1u_2(1+a(1-u_1^p)(1-u_2^p))$. We have $\Lambda = 1$ since $\sup \phi=1$ at $u=0$, and $\lambda=-1/p$ is attained at $u=1$. Hence $\pi =\Lambda/(\Lambda - \lambda)=p/(p+1)$, and the auxiliary cdfs are
$$F_{[0]}(u)=u-u(1-u^p)=u^{p+1},\qquad F_{[1]}(u)=u+u(1-u^p)/p=u\left(1+\frac{1-u^p}{p}\right).$$
The admissible range is $a \in [-\min(1/p^2,1), 1/p]$. This is exactly the stochastic representation found in \cite{blier-wong2024new}. Alternatively, let $V,W \sim \Unif(0,1)$ be independent, then $U_{[0]}\eqd V^{\frac{1}{p+1}}$ so $U_{[0]}\sim\mathrm{Beta}(p+1,1)$, and $U_{[1]}\eqd W V^{\frac{1}{p+1}}.$
\end{example}

\begin{example}[Huang--Kotz II]
Take $g_1(u)=g_2(u) = u(1-u)^q$, so that $C(u_1,u_2)=u_1u_2(1+a(1-u_1)^q(1-u_2)^q)$. We have $\phi_1(u)=\phi_2(u)=(1-u)^{q-1}(1-(q+1)u)$. Assume $q>1$ so that $\phi$ is bounded. We have $\sup\phi=1$ at $u=0$, hence $\Lambda=1$, and
$$\lambda=-\left(\frac{q+1}{q-1}\right)^{q-1}.$$
Therefore $\pi =\Lambda/\left(\Lambda-\lambda\right)=1/\left(1-\lambda\right)$. The auxiliary cdfs are
$$F_{[0]}(u)=u-\Lambda g(u)=u-u\left(1-u\right)^q=u\left(1-\left(1-u\right)^q\right),$$
$$F_{[1]}(u)=u-\lambda g(u)=u-\lambda u\left(1-u\right)^q=u\left(1-\lambda\left(1-u\right)^q\right).$$
In this symmetric case, the admissible range is $a\in\left[-1,|\lambda|\right]$. The stochastic representation is a new result for the Huang--Kotz II family. Also note that $U_{[0]}$ has a simple representation: if $V,W\sim\Unif(0,1)$ are independent and $Z:=1-V^{1/q}\sim\mathrm{Beta}\left(1,q\right)$, then $U_{[0]}\eqd \max\left(W,Z\right)$, since $\p\left(\max\left(W,Z\right)\le u\right)=u\left(1-\left(1-u\right)^q\right)=F_{[0]}(u)$. In contrast, $U_{[1]}$ does not seem to admit a comparably simple closed-form representation in terms of elementary transforms. One exception is the case $q=2$. Letting $W\sim\Unif(0,1)$, we have
$$U_{[1]}\eqd \frac{2}{3}+\sqrt[3]{\frac{W-\frac{8}{9}}{3}}.$$
\end{example}

\begin{example}[Bairamov--Kotz--Bekçi kernels]\label{ex:BKB}

Take $g_1(u) = u(1-u^{p_1})^{q_1}$ and $g_2(v) = v(1-v^{p_2})^{q_2}$, so that $C(u_1,u_2)=u_1u_2(1+a(1-u_1^{p_1})^{q_1}(1-u_2^{p_2})^{q_2})$. This family was proposed in \cite{bairamov2001new} as an extension of the Huang--Kotz I family. For $k\in\left\{1, 2\right\}$, we have
$$\phi_k(t)=g_k'(t)=\left(1-t^{p_k}\right)^{q_k-1}\left(1-\left(1+p_k q_k\right)t^{p_k}\right).$$ 
Assume $q_1>1$ and $q_2>1$ so that $\phi_1$ and $\phi_2$ are bounded. Then $\sup\phi_1=\sup\phi_2=1$, hence $\Lambda_1=\Lambda_2=1$, and
$$\lambda_k=-\frac{\left(1+p_k q_k\right)^{q_k-1}}{p_k^{q_k}\left(q_k-1\right)^{q_k-1}},\qquad k\in\left\{1, 2\right\}.$$
Therefore $\pi_1 = 1/\left(1-\lambda_1\right)$ and $\pi_2 =1/\left(1-\lambda_2\right)$, and the auxiliary cdfs are
$$F_{U,[0]}(u)=u\left(1-\left(1-u^{p_1}\right)^{q_1}\right),\qquad F_{U,[1]}(u)=u\left(1-\lambda_1\left(1-u^{p_1}\right)^{q_1}\right),$$
$$F_{V,[0]}(v)=v\left(1-\left(1-v^{p_2}\right)^{q_2}\right),\qquad F_{V,[1]}(v)=v\left(1-\lambda_2\left(1-v^{p_2}\right)^{q_2}\right).$$
The admissible range is 
$$a \in \left[-\min\left(1, \frac{\left(1+p_1 q_1\right)^{q_1-1}}{p_1^{q_1}\left(q_1-1\right)^{q_1-1}}
\frac{\left(1+p_2 q_2\right)^{q_2-1}}{p_2^{q_2}\left(q_2-1\right)^{q_2-1}}\right),
\min\left(\frac{\left(1+p_1 q_1\right)^{q_1-1}}{p_1^{q_1}\left(q_1-1\right)^{q_1-1}},
\frac{\left(1+p_2 q_2\right)^{q_2-1}}{p_2^{q_2}\left(q_2-1\right)^{q_2-1}}\right)\right].$$

The distribution $F_{U,[0]}$ admits an explicit representation. Let $V,W\sim\Unif(0,1)$ be independent and define $S:=\left(1-V^{1/q_1}\right)^{1/p_1}$. Then $\p\left(S\le u\right)=1-\left(1-u^{p_1}\right)^{q_1}$ for $u\in[0,1]$, and therefore $U_{[0]}\eqd \max\left(W,S\right)$, since $\p\left(\max\left(W,S\right)\le u\right)=u\left(1-\left(1-u^{p_1}\right)^{q_1}\right)=F_{U,[0]}(u)$. An analogous representation holds for $V_{[0]}$. In contrast, the cdfs $F_{U,[1]}$ and $F_{V,[1]}$ do not appear to admit simple representations.
\end{example}

\subsection{Extremal Spearman's rho and Kendall's tau}\label{subsec:extreme-rho}

This subsection quantifies the maximal amount of dependence, 
as measured by Spearman's $\rho_S$ and Kendall's $\tau$, 
that can be achieved within Sarmanov copulas in \eqref{eq:farlie-2d}. 
We identify sharp global bounds (over all admissible kernels 
and all admissible dependence parameters) and find the extremal kernels, 
which turn out to be of ``two-block'' (checkerboard) type.

The bivariate Spearman rank coefficient is
\begin{equation}\label{eq:spearman-def}
   \rho_S=\mathrm{Corr}(U_1,U_2)=12\int_{[0,1]^2} u_1 u_2\d C(u_1,u_2)-3 =12\int_{[0,1]^2}\left(C(u_1,u_2)-u_1 u_2\right)\d u_1 \d u_2.
\end{equation}
Kendall's $\tau$ is defined as the difference between 
the probabilities of concordance and discordance, that is, 
$\tau=\p\left((X-X')(Y-Y')>0\right)-\p\left((X-X')(Y-Y')<0\right)$, 
where $(X', Y')$ is an independent copy of $(X, Y)$. 
If $C$ admits a density $c$, then 
$$\tau=4\int_{[0,1]^2} C(u_1,u_2)c(u_1,u_2)\d u_1 \d u_2 - 1.$$

\begin{remark}
   In \cite{fredricks2007relationship}, the authors observe that copulas 
   for which Spearman's rho and Kendall's tau are linked 
   by the identity $3\tau=2\rho_S$ satisfy the partial differential equation
$$\frac{\partial^2}{\partial u_1 \partial u_2}\log\left\lvert C(u_1,u_2)-u_1u_2\right\rvert=0,$$ 
whenever $C(u_1,u_2)\neq u_1 u_2$. 
Sarmanov copulas satisfy this PDE; therefore, 
finding one measure of concordance suffices to determine the other. 
\end{remark}

\begin{proposition}\label{prop:rho-tau-separable}
Let $G_1:=\int_0^1 g_1$ and $G_2:=\int_0^1 g_2$. Then
\begin{equation}
\label{eq:rho-formula}
\rho_S(C)=12\int_0^1\int_0^1\left(C(u_1,u_2)-u_1 u_2\right)\d u_1 \d u_2 =12 aG_1G_2.
\end{equation}
\end{proposition}

\begin{proof}
We have that $C(u_1,u_2)-u_1u_2=a g_1(u_1)g_2(u_2)$, and applying the 
Fubini--Tonelli theorem leads to
$$\int_0^1\int_0^1 \left(C(u_1,u_2)-u_1 u_2\right)\d u_1 \d u_2 =a\left(\int_0^1 g_1(u_1)\d u_1\right)\left(\int_0^1 g_2(u_2)\d u_2\right)=aG_1G_2,$$
which proves \eqref{eq:rho-formula}.
\end{proof}

Because $\rho_S(C)=12a G_1G_2$, finding the extremal values of $\rho_S$ 
reduces to understanding (i) the admissible magnitude of $a$, 
and (ii) the largest possible values of the signed ``areas'' 
$G_1$ and $G_2$ under the slope constraints on $g_1$ and $g_2$. 
The following theorem identifies the global extrema of $\rho_S$ over all 
Sarmanov copulas of the form \eqref{eq:farlie-2d} under the absolute-continuity and bounded-derivative conditions stated in Section \ref{ss:farlie-def}. 
The proof is given in Appendix \ref{sec:proofs}. 
This result was previously shown in \cite{shubina2004maximuma}, 
but we re-derive it in a way that is more natural in our setting. 
Rather than optimizing over the admissible Sarmanov kernels directly,
we view the model through its stochastic representation 
and identify the sharp bound from the resulting structure. 

\begin{theorem}\label{thm:global-extrema-rho}
Let $C$ be a Sarmanov copula as in \eqref{eq:farlie-2d} with 
kernels $g_1, g_2$ satisfying the absolute-continuity and 
bounded-derivative conditions stated in Section \ref{ss:farlie-def}. Then
\begin{equation}
\label{eq:rho-global-bound}
-\frac{3}{4} \le \rho_S(C) \le \frac{3}{4}.
\end{equation}
The bounds are sharp. 
\end{theorem}

The extremal values in Theorem \ref{thm:global-extrema-rho} 
are attained by checkerboard kernels. 
These kernels identified above admit a convenient Bernoulli-mixture interpretation. 
Take $U_{1,[1]}\sim \mathrm{Unif}(0,1/2)$, $U_{1,[0]}\sim \mathrm{Unif}(1/2,1)$ 
and similarly for $U_2$, all independent, and let $(I_1, I_2)$ 
be a bivariate Bernoulli vector with $\p(I_1=1)=\p(I_2=1)=1/2$ and 
$\p(I_1=1,I_2=1)=(1+\theta)/4$. Then, $U_1 = (1-I_1)U_{1,[0]} + I_1 U_{1,[1]}$ 
and $U_2 = (1-I_2)U_{2,[0]} + I_2 U_{2,[1]}$ have uniform margins, 
their copula is
$$C_\theta(u_1, u_2)=u_1u_2+\frac{\theta}{4}\left(1-\left|1-2u_1\right|\right)\left(1-\left|1-2u_2\right|\right),$$
and $\rho_S(C_\theta)=3\theta/4$. 
The extremal values $\rho_S=\pm 3/4$ are attained at $\theta=\pm 1$.

\subsection{Tail dependence}\label{ss:tail-dependence}

For a bivariate copula $C$, define the lower and upper 
tail dependence coefficients (when the limits exist) by
$$\lambda_L(C):=\lim_{u\downarrow0}\frac{C(u,u)}{u}, \qquad \lambda_U(C):=\lim_{u\uparrow1}\frac{1-2u+C(u,u)}{1-u}.$$

\begin{proposition}\label{prop:tail-independence}
For any Sarmanov copula $C$ of the form \eqref{eq:farlie-2d} 
under the absolute-continuity and bounded-derivative conditions 
stated in Section \ref{ss:farlie-def}, one has 
$\lambda_L(C)=\lambda_U(C)=0,$ that is, the tail dependence parameter 
is identically zero (independently of the kernels).
\end{proposition}

\begin{proof}
Let $L_k:=\|g_k'\|_\infty = \max(-1/\lambda_k, 1/\Lambda_k) <\infty$ 
for $k\in\{1,2\}$. Since $g_k(0)=g_k(1)=0$ and $g_k$ is absolutely continuous,
$$|g_k(u)|=\left|\int_0^u g_k'(t)\d t\right|\le L_k u, \qquad |g_k(1-u)|=\left|\int_{1-u}^1 g_k'(t)\d t\right|\le L_k u,\qquad u\in[0,1].$$
Using $C(u,u)=u^2+ag_1(u)g_2(u)$, we have
$$0\le \frac{C(u,u)}{u} = u + a \frac{g_1(u)g_2(u)}{u} \le u + |a| \frac{(L_1u)(L_2u)}{u} = (1+|a|L_1L_2)u,$$
which tends to zero as $u\downarrow0$, so $\lambda_L(C)=0$. For the upper tail, set $\overline u =1-u$ and note 
$1-2u+C(u,u)=\overline u^2+a g_1(1-\overline u)g_2(1-\overline u)$, hence
$$ 0\le \frac{1-2u+C(u,u)}{1-u} = \overline u + a \frac{g_1(1-\overline u)g_2(1-\overline u)}{\overline u} \le \overline u + |a| \frac{(L_1\overline u)(L_2\overline u)}{\overline u} = (1+|a|L_1L_2)\overline u,$$
which tends to zero as $\overline u \downarrow 0$, so $\lambda_U(C)=0$.
\end{proof}

\section{Multivariate extensions of Sarmanov copulas}\label{sec:dvariate}

While the bivariate Sarmanov perturbations discussed 
in Section \ref{sec:farlie} already cover a wide range of asymmetric kernels, 
modern applications frequently require high-dimensional dependence models. 
In principle, one could write down a $d$-variate perturbation 
and then verify the $d$-increasing property directly. 
However, this is both conceptually and technically challenging. 
In practice, this becomes rapidly intractable: the $d$-increasing constraints 
are combinatorial (they correspond to nonnegativity of all $2^d$ 
inclusion--exclusion increments) and are typically tedious once $d\ge 3$. 
The key advantage of stochastic representations 
such as those in \cite{blier-wong2022stochastic,blier-wong2024new} 
is that verifying the $d$-increasing property is no longer required: 
it suffices to construct $\bm U$ as a random vector with uniform margins 
using a Bernoulli-mixing mechanism based on lower-dimensional kernels; 
the resulting joint cdf is automatically a valid copula. 
Moving from $d=2$ to general $d$ is conceptually trivial, 
as one replaces a bivariate Bernoulli index by 
a $d$-variate Bernoulli random vector.

\subsection{General multivariate construction}\label{subsec:dvariate-general}

The following theorem proposes a d-variate extension of the bivariate 
Sarmanov class in $d$ dimensions, based on a $d$-variate Bernoulli mixture. 
The principle is to construct a random vector $\bm U$ 
with standard uniform margins by mixing two one-dimensional 
components in each coordinate, and encode all dependence 
with a latent multivariate Bernoulli vector. 
Since $\bm U$ has uniform margins, its joint cdf is automatically a copula.

We first introduce the notion of a $\pi$-calibrated pair of cdfs, 
which characterizes the auxiliary cdfs in Theorem \ref{thm:stochastic}.
\begin{definition}[$\pi$-calibrated pair]
    Fix $\pi\in(0,1)$. Two cdfs $F_{[0]}$ and $F_{[1]}$ 
    supported on $[0,1]$ are called $\pi$-calibrated 
    if their mixture is uniform, that is, they satisfy
\begin{equation}\label{eq:calibr}
    (1-\pi)F_{[0]}(u)+\pi F_{[1]}(u)=u,\qquad \forall u\in[0,1].
\end{equation}
\end{definition}
Let $I\sim\mathrm{Bernoulli}(\pi)$ be independent of $(U_{[0]},U_{[1]})$ with $U_{[0]}\sim F_{[0]}$ and $U_{[1]}\sim F_{[1]}$. If $(F_{[0]},F_{[1]})$ is a $\pi$-calibrated pair, then the mixture $(1-I)U_{[0]}+IU_{[1]}$ is uniformly distributed on $[0,1]$. 

In what follows, we let $\Pi(\bm u)=\prod_{m=1}^d u_m$ denote the independence copula in $d$ dimensions. Let $d\ge 2$ and write $[d]:=\{1,\dots,d\}$. Let $\bm I=(I_1,\dots,I_d)\in\{0,1\}^d$ be a multivariate Bernoulli vector with $\p(I_m=1)=:\pi_m\in(0,1), m\in[d].$ For each $m\in[d]$, choose a $\pi_m$-calibrated pair $(F_{m,[0]},F_{m,[1]})$ supported on $[0,1]$, and let $U_{m,[0]}\sim F_{m,[0]}$ and $U_{m,[1]}\sim F_{m,[1]}.$ Assume the collection $\{U_{m,[0]},U_{m,[1]}:m\in[d]\}$ is independent and independent of $\bm I$. Define 
\begin{equation}\label{eq:unif-mix}
   U_m := (1-I_m)U_{m,[0]} + I_m U_{m,[1]}, m\in[d],
\end{equation}
and write $\bm U=(U_1,\dots,U_d)$. For later use, define the cdf gaps and kernel factors
$$\Delta_m(u):=F_{m,[1]}(u)-F_{m,[0]}(u); \quad g_m(u):= \pi_m \Delta_m(u),\qquad u\in[0,1].$$
For each subset $S\subseteq[d]$ with $|S|\ge 2$, define the normalized mixed moments
\begin{equation}\label{eq:thetas}
    \theta_S :=\E\left[\prod_{m\in S}\frac{I_m-\pi_m}{\pi_m}\right].
\end{equation}

\begin{theorem}\label{thm:dvariate-bernoulli}
    Let $\bm U =(U_1,\dots,U_d)$ be defined as in \eqref{eq:unif-mix}. Then, $\bm U$ has copula  $C$ that satisfies, for all $\bm u=(u_1,\dots,u_d)\in[0,1]^d$,
    \begin{equation}\label{eq:cop-d}
        C(\bm u)=\Pi(\bm u) + \sum_{\substack{S\subseteq[d]\\ |S|\ge 2}}\theta_S \left(\prod_{m\notin S} u_m \right)\left(\prod_{m\in S} g_m(u_m)\right).
    \end{equation}
\end{theorem}

\begin{proof}
Fix $m\in[d]$. Since $U_m=(1-I_m)U_{m,[0]}+I_m U_{m,[1]}$ and $I_m$ is independent of $(U_{m,[0]},U_{m,[1]})$, we have that $\p(U_m\le u)=(1-\pi_m)F_{m,[0]}(u)+\pi_m F_{m,[1]}(u)=u$ by the $\pi_m$-calibration condition \eqref{eq:calibr}. Hence $U_m\sim\mathrm{Unif}(0,1)$ for every $m$. For $u\in[0,1]^d$, conditional on $\bm I$ the coordinates are independent, so
\begin{equation}\label{eq:copula-d-expected}
    C(\bm u)=\p(\bm U\le \bm u) =\E\left[\prod_{m=1}^d F_{m,[I_m]}(u_m)\right].
\end{equation}
Let $Z_m:=(I_m-\pi_m)/\pi_m$ so that $\E[Z_m]=0$. From \eqref{eq:calibr} and $\Delta_m=F_{m,[1]}-F_{m,[0]}$, we have $F_{m,[0]}(u)=u-\pi_m\Delta_m(u)$ and $F_{m,[1]}(u)=u+(1-\pi_m)\Delta_m(u),$ and therefore for $I_m\in\{0,1\}$,
\begin{equation}\label{eq:cdf-from-z}
   F_{m,[I_m]}(u)=u+(I_m-\pi_m)\Delta_m(u)=u+g_m(u)Z_m.
\end{equation}
Substituting into \eqref{eq:copula-d-expected} yields
$$C(\bm u)=\E\left[\prod_{m=1}^d \left(u_m+g_m(u_m)Z_m\right)\right].$$

Expand the product over subsets $S\subseteq[d]$ to get 
$$\prod_{m=1}^d(u_m+g_m(u_m) Z_m)= \Pi(\bm u) + \sum_{\substack{S\subseteq[d]\\ |S|\ge 1}}\left(\prod_{m\notin S} u_m\right)\left(\prod_{m\in S} g_m(u_m)\right)\left(\prod_{m\in S} Z_m\right).$$
The $|S|=0$ term is $\Pi(\bm u)$, and the $|S|=1$ terms vanish since $\mathbb E[Z_m]=0$. Taking expectations yields \eqref{eq:cop-d} with coefficients $\theta_S$ defined in \eqref{eq:thetas}. 
\end{proof}

The representation \eqref{eq:cop-d} yields a generalized FGM-type subset expansion, with dependence parameters \eqref{eq:thetas} given by centred mixed moments of the Bernoulli vector $\bm I$, extending the bivariate covariance-type parameter in \eqref{eq:theta-param}. When $d=2$, \eqref{eq:cop-d} reduces exactly to the bivariate Sarmanov form \eqref{eq:farlie-2d}. For $d\ge 3$, higher-order interaction terms are present in general; the copula in \eqref{eq:cop-d} is automatically valid as long as the parameters $\theta_S$ are compatible with a valid Bernoulli pmf for $\bm I$ and the kernels $g_m$ arise from $\pi_m$-calibrated pairs. 

\subsection{Trivariate Sarmanov copulas}\label{subsec:trivariate}

For $d=3$, the subset expansion becomes particularly transparent and already illustrates the emergence of pairwise and three-way interaction terms. The following is the direct specialization of Theorem \ref{thm:dvariate-bernoulli}. 

\begin{corollary}\label{cor:trivariate}
Let $(I_1,I_2,I_3)\in\{0,1\}^3$ be a trivariate Bernoulli vector with $\p(I_1=1)=\pi_1$, $\p(I_2=1)=\pi_2$, and $\p(I_3=1)=\pi_3$. Construct $(U_1,U_2,U_3)$ through \eqref{eq:unif-mix} as in Theorem \ref{thm:dvariate-bernoulli}, with kernels $g_1,g_2,g_3$. Then the copula of $(U_1,U_2,U_3)$ is
\begin{align}\label{eq:trivariate-copula}
C(u_1,u_2,u_3) &=u_1 u_2 u_3 +\theta_{12}g_1(u_1)g_2(u_2)u_3 +\theta_{13}g_1(u_1)u_2g_3(u_3) +\theta_{23}u_1g_2(u_2)g_3(u_3)\nonumber\\
&\qquad\qquad\qquad +\theta_{123}g_1(u_1)g_2(u_2)g_3(u_3),
\end{align}
where 
$$
\begin{aligned}
\theta_{12} &= \frac{\E\left[(I_1-\pi_1)(I_2-\pi_2)\right]}{\pi_1 \pi_2}, &\qquad
\theta_{13} &= \frac{\E\left[(I_1-\pi_1)(I_3-\pi_3)\right]}{\pi_1 \pi_3},\\[0.8ex]
\theta_{23} &= \frac{\E\left[(I_2-\pi_2)(I_3-\pi_3)\right]}{\pi_2 \pi_3},
&\qquad \theta_{123} &= \frac{\E\left[(I_1-\pi_1)(I_2-\pi_2)(I_3-\pi_3)\right]}{\pi_1 \pi_2 \pi_3}.
\end{aligned}$$
It follows that one never needs to check the $3$-increasing property directly: \eqref{eq:trivariate-copula} is a copula because $(U_1, U_2, U_3)$ is a random vector with standard uniform margins.
\end{corollary}

We can therefore construct a wide variety of trivariate Sarmanov copulas by choosing appropriate $\pi$-calibrated pairs in each margin, and an appropriate trivariate Bernoulli pmf for $(I_1, I_2, I_3)$. As a concrete example, we now construct a trivariate extension of the bivariate BKB family in Example \ref{ex:BKB}. Let $g_1(u)=u(1-u^{p_1})^{q_1}$, $g_2(u)=u(1-u^{p_2})^{q_2}$, and $g_3(u)=u(1-u^{p_3})^{q_3}$, which yields
\begin{align}\label{eq:BK-trivariate}
C(u_1,u_2,u_3) =u_1u_2u_3\big(1 
&+\theta_{12}(1-u_1^{p_1})^{q_1}(1-u_2^{p_2})^{q_2}
+\theta_{13}(1-u_1^{p_1})^{q_1}(1-u_3^{p_3})^{q_3}\nonumber\\
&+\theta_{23}(1-u_2^{p_2})^{q_2}(1-u_3^{p_3})^{q_3}
+\theta_{123}(1-u_1^{p_1})^{q_1}(1-u_2^{p_2})^{q_2}(1-u_3^{p_3})^{q_3}\big).
\end{align}
The parameters $\theta_{12},\theta_{13},\theta_{23},\theta_{123}$ must be compatible with a valid trivariate Bernoulli pmf for $(I_1,I_2,I_3)$, which is easily verified using some results of \cite{teugels1990representations}. 

\subsection{Exchangeable Sarmanov copulas}\label{subsec:exchangeable-dvariate}

Many applications require the dependence model to be exchangeable, that is, invariant under permutations of the coordinates.
Within the Bernoulli-mixture construction of Theorem \ref{thm:dvariate-bernoulli}, exchangeability (and, when desired, radial symmetry)
can be imposed through simple symmetry constraints on (i) the marginal $\pi$-calibrated pairs and (ii) the latent Bernoulli index vector.

\begin{definition}\label{def:exchangeable-radial}
A $d$-copula $C$ is exchangeable if
$$C(u_1,\dots,u_d)=C\left(u_{\sigma(1)},\dots,u_{\sigma(d)}\right)$$
for all permutations $\sigma$ of $[d]$. 

A random vector $\bm X=(X_1,\dots,X_d)$ with continuous margins is radially symmetric about $\bm a\in\mathbb R^d$ if $\bm X-\bm a\eqd \bm a-\bm X$. For copulas, we only need the case $\bm a=\frac12\bm 1$, for which radial symmetry reduces to $\bm U\eqd \bm 1-\bm U$ componentwise. In the bivariate case, if $C$ is the copula of $(U_1,U_2)$, radial symmetry about $(1/2,1/2)$ is equivalent to $C(u_1,u_2)=u_1+u_2-1+C(1-u_1,1-u_2)$ for all $(u_1,u_2)\in[0,1]^2$; see, for instance, \cite[Section 2.7]{nelsen2006introduction}.
\end{definition}

\begin{proposition}\label{prop:exchangeable-radial-farlie}
Let $\bm U$ be constructed as in Theorem \ref{thm:dvariate-bernoulli}.
\begin{enumerate}
   \item Assume that the calibrated pairs are identical across margins, that is,
$$\pi_m = \pi; \quad (F_{m,[0]}, F_{m,[1]}) = (F_{[0]},F_{[1]})$$ 
for $m=1,\dots,d$, and that the Bernoulli index vector $\bm I$ is exchangeable. Then the copula of $\bm U$ is exchangeable and can be written in the reduced form
\begin{equation}\label{eq:exchangeable-farlie}
C(\bm u)
= \Pi(\bm u) +\sum_{k=2}^d \theta_k
\sum_{\substack{S\subseteq[d]\\|S|=k}}\left(\prod_{m\notin S} u_m\right)\left(\prod_{m\in S} g(u_m)\right),
\end{equation}
where $g(u)=\pi(F_{[1]}(u)-F_{[0]}(u))$ and, for any $k=2,\dots,d$,
\begin{equation}\label{eq:theta-k-exchangeable}
\theta_k=\E\left[\prod_{m=1}^k \frac{I_m-\pi}{\pi}\right],
\end{equation}
which is well-defined (does not depend on the choice of $k$ indices) by exchangeability of $\bm I$.
   \item Assume furthermore that $\pi=\frac12$, that the pair $(F_{[0]},F_{[1]})$ satisfies the reflection identity
\begin{equation}\label{eq:reflection-calibrated-pair}
F_{[1]}(u)=1-F_{[0]}(1-u),\qquad u\in[0,1],
\end{equation}
and that the Bernoulli vector is palindromic in the sense that $\bm I\eqd \bm 1-\bm I$. Then $\bm U$ is radially symmetric about $\frac12\bm 1$. With $Z_m:=2I_m-1\in\{-1,1\}$ one has
$\theta_k=\E[Z_1\cdots Z_k]$ and $\theta_k=0$ for all odd $k$.
\end{enumerate}

\end{proposition}

\begin{proof}
Part (i) follows from Theorem \ref{thm:dvariate-bernoulli} and the fact that, under the stated assumptions, all mixed moments $\theta_S=\E[\prod_{m\in S}(I_m-\pi)/\pi]$ depend on $S$ only through $|S|$.

For part (ii), \eqref{eq:reflection-calibrated-pair} is equivalent to $U_{m,[1]}\eqd 1-U_{m,[0]}$ for each $m$. Hence
$$1-U_m=(1-I_m)(1-U_{m,[0]})+I_m(1-U_{m,[1]})\eqd (1-I_m)U_{m,[1]}+I_mU_{m,[0]}.$$
Writing $I_m':=1-I_m$, the right-hand side equals $(1-I_m')U_{m,[0]}+I_m'U_{m,[1]}$, that is, it is distributed as $U_m$ constructed with index $I_m'$ instead of $I_m$. Therefore $\bm 1-\bm U$ has the same distribution as the construction in Theorem \ref{thm:dvariate-bernoulli} with index vector $\bm I'=\bm 1-\bm I$, and the claim $\bm U\eqd \bm 1-\bm U$ follows from $\bm I'\eqd \bm I$.
Finally, when $\pi=\frac12$, we have $\theta_k=\E[Z_1\cdots Z_k]$; symmetry implies $(Z_1,\dots,Z_d)\eqd-(Z_1,\dots,Z_d)$,
hence $\theta_k=(-1)^k\theta_k$, so $\theta_k=0$ for odd $k$.
\end{proof}

\begin{remark}\label{rem:q-half-kernel-shape}
In Theorem \ref{thm:stochastic}, the marginal Bernoulli success probabilities are
$\pi_1=\Lambda_1/(\Lambda_1-\lambda_1)$ and $\pi_2=\Lambda_2/(\Lambda_2-\lambda_2)$.
Thus $\pi_1=\frac12$ if and only if $\Lambda_1=-\lambda_1$, equivalently
$\esssup_{[0,1]}\phi_1=-\essinf_{[0,1]}\phi_1$, and similarly for $\phi_2$.
A convenient sufficient structural condition is symmetry of the kernel about $1/2$,
\begin{equation}\label{eq:symmetric-lift}
g_1(1-u)=g_1(u)\quad\text{for all }u\in[0,1]
\end{equation}
which forces $\phi_1(1-u)=-\phi_1(u)$ and hence $\Lambda_1=-\lambda_1$. When \eqref{eq:symmetric-lift} holds (and $\Lambda_1=-\lambda_1$), the auxiliary cdfs in Theorem \ref{thm:stochastic} satisfy the reflection identity $F_{U,[1]}(u)=1-F_{U,[0]}(1-u)$, so the two mixture components are mirror images around $1/2$. This includes the classical FGM kernel $g(u)=u(1-u)$, but excludes the HKI/HKII kernels unless the exponent reduces to the FGM case.
\end{remark}

\begin{remark}\label{rem:exch-sym-bernoulli}
If $\bm I$ is exchangeable with $\p(I_m=1)=\pi$, its distribution is characterized by the distribution of the sum $S_d:=\sum_{m=1}^d I_m$. Writing $w_j:=\p(S_d=j)$, one has $\p(\bm I=\bm i)=w_{|\bm i|}/\binom{d}{|\bm i|}$ for $\bm i\in\{0,1\}^d,$ so the admissibility constraints reduce to $w_j\ge0$ and $\sum_{j=0}^d w_j=1$. In the symmetric case where $\pi=\frac12$, the additional requirement $\bm I\eqd \bm 1-\bm I$ is equivalent to $w_j=w_{d-j}$, $j=0,\dots,d$, and implies that all odd-order coefficients $\theta_k$ vanish in \eqref{eq:exchangeable-farlie}. This is precisely the latent structure underlying the exchangeable FGM copulas studied in \cite{blier-wong2024exchangeable}. The extremal exchangeable symmetric Bernoulli distributions (hence extremal copulas in the supermodular order) correspond to the two-point ``ray'' distributions for $S_d$ (see, for instance, \cite[Proposition 4.5]{fontana2021model}). We return to END/EPD bounds and their effect on $\rho_d^\pm$ in Subsection \ref{subsec:d-variate-rho}.
\end{remark}

\subsection{Multivariate orthant Spearman's rho and extremal dependence}\label{subsec:d-variate-rho}

Let $C$ be a $d$-variate copula of the form \eqref{eq:cop-d} in Theorem \ref{thm:dvariate-bernoulli}. Define the kernel integrals
$$\kappa_m:=\int_0^1 g_m(u)\d u, \qquad \kappa_S:=\prod_{m\in S}\kappa_m.$$

The average upper- and lower-orthant Spearman coefficients were proposed in \cite{nelsen1996nonparametric} and further studied in \cite{gijbels2021specification}. They are defined as
\begin{align}
\rho_d^+(C)
&=\frac{d+1}{2^d-(d+1)}\left(2^d\int_{[0,1]^d}\Pi(\bm u)\d C(\bm u)-1\right),\label{eq:rho-plus}\\
\rho_d^-(C)
&=\frac{d+1}{2^d-(d+1)}\left(2^d\int_{[0,1]^d}C(\bm u)\d \bm u-1\right).\label{eq:rho-minus}
\end{align}
Within the class of d-variate Sarmanov copulas introduced in this section, these parameters can be computed in closed form as follows.
\begin{proposition}\label{prop:rho-d-closed-form}
Let $C$ be given by \eqref{eq:cop-d}. Then
\begin{align}
\rho_d^-(C)
&=\frac{d+1}{2^d-(d+1)}\sum_{\substack{S\subseteq[d]\\|S|\ge2}} 2^{|S|} \theta_S \kappa_S,
\label{eq:rho-d-minus-farlie}\\
\rho_d^+(C)
&=\frac{d+1}{2^d-(d+1)}\sum_{\substack{S\subseteq[d]\\|S|\ge2}} (-1)^{|S|}2^{|S|} \theta_S \kappa_S.
\label{eq:rho-d-plus-farlie}
\end{align}
If all odd-order parameters are zero (that is, $\theta_S=0$ for odd $|S|$), then $\rho_d^+(C)=\rho_d^-(C)$.
\end{proposition}

\begin{proof}
We start with $\rho_d^-$. From \eqref{eq:cop-d} and the Fubini--Tonelli theorem, we have
$$\int_{[0,1]^d}C(\bm u)\d\bm u =\int_{[0,1]^d} \Pi(\bm u) \d \bm u +\sum_{|S|\ge2}\theta_S
\left(\prod_{m\notin S}\int_0^1 u_m\d u_m\right)
\left(\prod_{m\in S}\int_0^1 g_m(u_m)\d u_m\right).$$
Hence 
$$\int_{[0,1]^d}C(\bm u)\d\bm u = 2^{-d}+\sum_{|S|\ge2}\theta_S 2^{-(d-|S|)}\kappa_S,$$
and \eqref{eq:rho-d-minus-farlie} follows since
$$2^d\int_{[0,1]^d}C(\bm u)\d\bm u-1=\sum_{|S|\ge2}2^{|S|}\theta_S\kappa_S.$$

For $\rho_d^+$, use the stochastic representation in Theorem 
\ref{thm:dvariate-bernoulli}. Conditional on $I_m$, the components are independent and we have from \eqref{eq:cdf-from-z} that
$$\E[U_m\mid I_m] = 
\int_0^1 1 - F_{m,[I_m]}(u)\d u = \int_0^1 1 - u - g_m(u) Z_m \d u = \frac12-\kappa_m\frac{I_m-\pi_m}{\pi_m}.$$
Therefore, by expanding and taking expectations, we have
$$\E\left[\prod_{m=1}^d U_m\right]=\E\left[\prod_{m=1}^d \left(\frac12-\kappa_m\frac{I_m-\pi_m}{\pi_m}\right)\right]=2^{-d}+\sum_{|S|\ge2}(-1)^{|S|}2^{-(d-|S|)}\theta_S\kappa_S;$$
the last equality holds because $\E[Z_m] = 0$ for each $m$, thus all first order terms have zero expectation. Finally, \eqref{eq:rho-d-plus-farlie} follows from $\int\Pi(\bm u)\d C(\bm u)=\E[\prod_m U_m]$.
\end{proof}

If we fix the kernels $(g_1,\dots,g_d)$, equivalently the $\pi_m$-calibrated pairs $(F_{m,[0]},F_{m,[1]})_{m=1}^d$, we can study the effect of changing the Bernoulli index vector $\bm I$ on the induced copula $C$ and the dependence parameters $\rho_d^\pm$. The relevant stochastic order here is the supermodular order $\preceq_{\rm sm}$ (see, for instance, \cite{muller2002comparison,shaked2007stochastic}). Let $\bm I$ and $\bm I'$ be Bernoulli vectors with the same margins $(\pi_1,\dots,\pi_d)$, and construct $(U_1,\dots,U_d)$ and $(U_1',\dots,U_d')$ as in Theorem \ref{thm:dvariate-bernoulli}, using the same kernels but different index vectors. The following theorem is adapted from \cite[Theorem 4.2]{blier-wong2022stochastic} and stated without proof. 

\begin{theorem}\label{thm:sm-order-fixed-kernels}
Assume that $g_m$ has no sign-changes on $(0,1)$ for each $m\in[d]$. 
Let $\bm U$ and $\bm U'$ be constructed as in Theorem \ref{thm:dvariate-bernoulli} 
using the same kernels but different Bernoulli index vectors 
$\bm I$ and $\bm I'$ with identical margins. 
If $\bm I\preceq_{\rm sm} \bm I'$, then $\bm U\preceq_{\rm sm} \bm U'$. 
Consequently, $\rho_d^+(C_{\bm I})\le \rho_d^+(C_{\bm I'})$ and 
$\rho_d^-(C_{\bm I})\le \rho_d^-(C_{\bm I'})$. 
Among all Bernoulli vectors with fixed margins, 
the comonotonic coupling $\bm I^+$ is maximal 
in $\preceq_{\rm sm}$; hence $C_{\bm I^+}$ is the 
upper supermodular bound for fixed kernels and maximizes $\rho_d^\pm$.
\end{theorem}

The previous theorem assumes that the kernels $g_m$ 
have no sign-changes on $(0,1)$. 
This is satisfied by most standard choices of $\pi$-calibrated pairs, 
including the classical FGM kernel $g(u)=u(1-u)$, 
the HKI/HKII kernels, and the BKB kernels. 
We give an example of a kernel with sign-changes in Appendix \ref{sec:examples}.

Assume the exchangeable symmetric setting and fix a common kernel. 
Then there exist exchangeable vectors $\bm I^-$ and $\bm I^+$ such that, 
for all exchangeable $\bm I$, we have 
$\bm I^- \preceq_{\rm sm} \bm I \preceq_{\rm sm} \bm I^+$, 
and the extremal copulas are called the extreme negative dependence (END) 
and extreme positive dependence (EPD) copulas 
within that class of copulas. 
In the exchangeable Farlie setting with a common kernel integral 
$\kappa:=\int_0^1 g(u)\d u$, the closed forms in 
Proposition \ref{prop:rho-d-closed-form} simplify to
$$\rho_d^-(C)=\frac{d+1}{2^d-(d+1)}\sum_{k=2}^d \binom{d}{k}(2\kappa)^k \theta_k, \qquad \rho_d^+(C)=\frac{d+1}{2^d-(d+1)}\sum_{k=2}^d \binom{d}{k}(-1)^k(2\kappa)^k \theta_k.$$
For EPD and END one has $\theta_k=0$ for odd $k$, 
so $\rho_d^+(C)=\rho_d^-(C)$ in both cases.

For fixed $\bm I$, Proposition \ref{prop:rho-d-closed-form} shows that 
$\rho_d^\pm$ depends on the kernels only through the 
products $\kappa_S=\prod_{m\in S}\kappa_m$.
Thus, to maximize the magnitude of $\rho_d^\pm$, 
one should maximize $|\kappa_m|$ in each margin. 
In the symmetric Bernoulli-splitting construction ($\pi_m=1/2$), 
the sharp bound is $|\kappa_m|\le 1/4$, and equality is attained by 
the checkerboard kernels discussed in the proof of 
Theorem \ref{thm:global-extrema-rho}.

For $d=3$, Proposition \ref{prop:rho-d-closed-form} gives
$$\rho_3^-(C)=4(\theta_{12}\kappa_1\kappa_2+
\theta_{13}\kappa_1\kappa_3+\theta_{23}\kappa_2\kappa_3)+
8\theta_{123}\kappa_1\kappa_2\kappa_3;$$
$$\rho_3^+(C)=4(\theta_{12}\kappa_1\kappa_2+
\theta_{13}\kappa_1\kappa_3+\theta_{23}\kappa_2\kappa_3)-
8\theta_{123}\kappa_1\kappa_2\kappa_3.$$
In the exchangeable and symmetric case ($\kappa_1=\kappa_2=\kappa_3=\kappa$), 
END has $(\theta_2,\theta_3)=(-1/3,0)$ and 
EPD has $(\theta_2,\theta_3)=(1,0)$, so
$$\rho_3^{\pm}({\rm EPD})=12\kappa^2,\qquad \rho_3^{\pm}({\rm END})=-4\kappa^2.$$
For the classical FGM kernel $g(u)=u(1-u)$, 
one has $\kappa=\int_0^1 u(1-u) \d u=1/6$, 
hence $\rho_3^\pm\in[-1/9, 1/3]$.
With the checkerboard kernels one has $\kappa=1/4$, 
hence $\rho_3^\pm\in[-1/4, 3/4]$ in the exchangeable END--EPD range.

\section{Power transformations of Farlie copulas}\label{sec:power}

Motivated by the powered generalizations of the FGM family 
proposed in \cite{bekrizadeh2012new,bekrizadeh2022generalized}, 
we consider copulas obtained by taking a power of Farlie copulas. 
It is convenient to introduce the normalized kernels $(h_1,h_2)$ 
defined implicitly by
\begin{equation}\label{eq:h-def}
g_1(u_1)=u_1 h_1(u_1),\qquad g_2(u_2)=u_2 h_2(u_2),\qquad (u_1,u_2)\in[0,1]^2,
\end{equation}
(with the usual continuous extension at $0$ when needed). 
With this notation, \eqref{eq:farlie-2d} is equivalently the 
Farlie-type copula proposed in \cite{farlie1960performance}, given by
\begin{equation}\label{eq:farlie-h}
C(u_1,u_2)=u_1 u_2 \left(1+a h_1(u_1) h_2(u_2)\right).
\end{equation}
We refer the reader to Appendix \ref{app:comparison-three-classes} 
for the difference between Farlie and Sarmanov copulas. 
The powered family studied in \cite{bekrizadeh2012new,bekrizadeh2022generalized} is then
\begin{equation}\label{eq:powered}
C_{a,r}(u_1,u_2) =u_1 u_2 \left(1+a h_1(u_1) h_2(u_2)\right)^r,\qquad r \in \mathbb N.
\end{equation}
Direct verification that \eqref{eq:powered} is 
2-increasing is typically difficult. 
Our main observation in this section is that, 
when $r\in\mathbb{N}$, the powered copula \eqref{eq:powered} 
admits a simple block-maximum construction: 
it is the copula of componentwise maxima of an unpowered 
(that is, $r=1$) Farlie or Sarmanov copula with suitably transformed kernels. 
The main idea is to combine the Bernoulli-mixture representation of 
Theorem \ref{thm:stochastic} with a block-maximum construction that 
produces the desired power transform. 

\subsection{Block maxima and outer-power transformations}\label{subsec:power:blockmax}

For a copula $C$ and $r>1$, define the outer-power transform
\begin{equation}\label{eq:outer-power}
C^{(r)}(u_1,u_2):=C\left(u_1^{1/r},u_2^{1/r}\right)^{r},\qquad (u_1,u_2)\in[0,1]^2,
\end{equation}
which is a copula whenever $C$ on the right-hand side is a copula; 
see, for instance, \cite{nelsen2006introduction,liebscher2008construction}. 
The following proposition is well-known (see, for instance, 
\cite[Section 3.3.3]{nelsen2006introduction}), 
but we include a proof for completeness.
\begin{proposition}[Block maxima]\label{prop:block-maxima}
Let $(X_i,Y_i)_{i=1}^n$ be i.i.d. with continuous margins and copula $C$.
Define $M_X:=\max_{1\le i\le n}X_i$ and $M_Y:=\max_{1\le i\le n}Y_i$.
Then the copula of $(M_X,M_Y)$ equals $C^{(n)}$ as in \eqref{eq:outer-power}.
\end{proposition}

\begin{proof}
Let $F_X,F_Y$ be the margins of $(X_1,Y_1)$ and set 
$u_1=F_X(x)$, $u_2=F_Y(y)$. Then,
$$\p(M_X\le x,M_Y\le y) =\prod_{i=1}^n\p(X_i\le x,Y_i\le y) =C(u_1,u_2)^{n}.$$
The marginal probabilities satisfy $\p(M_X\le x)=u_1^{n}$ and 
$\p(M_Y\le y)=u_2^{n}$. Therefore,
$$\p(M_X\le x,M_Y\le y)=C(u_1,u_2)^{n} =C^{(n)}(u_1^n,u_2^n),$$
which proves the claim.
\end{proof}

The main observation in this section is that outer-power transforms 
are naturally connected to block maxima, 
which allows us to construct powered Farlie copulas 
through Bernoulli-mixture representations combined with the 
block-maxima representation. Fix kernels $h_1,h_2$ and $a\in\mathbb R$ 
and consider the Farlie copula
$$C_a(u_1,u_2)=u_1 u_2 (1+a h_1(u_1) h_2(u_2)).$$
For an integer $r\ge1$, define the transformed kernels 
$h_1^{\langle r\rangle}(x):=h_1(x^r)$ and $h_2^{\langle r\rangle}(y):=h_2(y^r),$ 
and the base copula
\begin{equation*}
C_a^{\langle r\rangle}(x,y):=xy(1+a h_1(x^r) h_2(y^r)).
\end{equation*}
Then its outer-power transform is the powered copula
\begin{equation}\label{eq:powered-bk}
C_{a,r}(u_1,u_2):=C_a^{\langle r\rangle}\left(u_1^{1/r},u_2^{1/r}\right)^{r}
=u_1 u_2 (1+a h_1(u_1) h_2(u_2))^{r}.
\end{equation}
Hence $C_{a,r}$ is the copula of maxima over $k=1,\dots,r$ 
taken separately in each component of $r$ i.i.d. samples 
from $C_a^{\langle r\rangle}$. 

We focus on integer powers $r\in\mathbb N$ in the following, 
as this case admits a straightforward stochastic representation 
through block maxima. 
However, the outer-power transform in \eqref{eq:outer-power} 
is well-defined for any real $r>1$. 
If the base copula is max-infinitely divisible 
(see, for instance, \cite{balkema1977max, resnick1987extreme, genest2018classa}), 
then the powered construction also extends to $0 < r < 1$. 

\subsection{Bernoulli-mixture representation for integer powers}\label{subsec:power:bernoulli}

We now show that the powered Farlie copulas in \eqref{eq:powered-bk} 
can be constructed through a Farlie representation combined with block maxima. 
This extends Theorem \ref{thm:dvariate-bernoulli} to the powered setting.
 
\begin{theorem}\label{thm:block-maxima}
Fix $r\in\mathbb{N}$. Let $(X_k,Y_k)_{k=1}^r$ be i.i.d. with copula
$C^{\langle r\rangle}_a(x,y)=xy(1+ah_1(x^r) h_2(y^r))$. Define
$$U_1:=\left(\max_{1\le k\le r}X_k\right)^r,\qquad 
  U_2:=\left(\max_{1\le k\le r}Y_k\right)^r.$$
Then $(U_1,U_2)$ has copula
$$C_{a,r}(u_1,u_2)=u_1 u_2 (1+a h_1(u_1) h_2(u_2))^r,\qquad (u_1,u_2)\in[0,1]^2.$$
\end{theorem}

\begin{proof}
It suffices to find the copula of $(U_1, U_2)$. 
Let $(X,Y)$ have uniform margins and copula
$$C^{\langle r\rangle}_a(x,y)=xy(1+a h_1(x^r) h_2(y^r)).$$
Take i.i.d. copies $(X_k,Y_k)_{k=1}^r$ and set 
$M_X=\max_{1\le k\le r}X_k$, $M_Y=\max_{1\le k\le r}Y_k$, 
and let $U_1=M_X^r$ and $U_2=M_Y^r$. 
Then for $u\in[0,1]$,
$$\p(U_1\le u)=\p(M_X^r\le u)=\p(M_X\le u^{1/r})=\p(X\le u^{1/r})^r=u,$$
so $U_1\sim\Unif(0,1)$ (and similarly $U_2\sim\Unif(0,1)$). 
For $(u_1,u_2)\in[0,1]^2$,
\begin{align*}
\p(U_1\le u_1,U_2\le u_2) &=\p(M_X\le u_1^{1/r}, M_Y\le u_2^{1/r}) 
                           =\p(X\le u_1^{1/r}, Y\le u_2^{1/r})^r \\
                          &=(C^{\langle r\rangle}_a(u_1^{1/r},u_2^{1/r}))^r =\left(u_1^{1/r}u_2^{1/r}(1+a h_1(u_1) h_2(u_2))\right)^r \\
                          &=u_1 u_2 (1+a h_1(u_1) h_2(u_2))^r =C_{a,r}(u_1,u_2),
\end{align*}
as claimed.
\end{proof}

\begin{remark}\label{rem:a-range-transformed-kernel}
The admissible range of the dependence parameter for the powered copula 
$C_{a,r}(u_1,u_2)=u_1 u_2 (1+a h_1(u_1) h_2(u_2))^r$ is not the same 
as the admissible range for the base copula 
$C_{a,1}(u_1, u_2)=u_1 u_2 (1+ah_1(u_1) h_2(u_2))$. 
Theorem \ref{thm:block-maxima} provides a sufficient condition, 
namely, that if the auxiliary copula 
$C^{\langle r\rangle}_a(u_1, u_2) = u_1u_2 (1 + a h_1(u_1^r) h_2(u_2^r))$ 
is a copula, then the powered copula $C_{a,r}$ is also a copula. 
Thus the admissible range obtained from the transformed kernels 
$h_1(x^r)$ and $h_2(y^r)$ yields a sufficient condition for validity, 
and the inclusion can be strict.

For example, if $h_1(x)=1-x$ and $h_2(u)=1-u$ then 
$g_1(x)=x(1-x^r)$ and $\phi_1(x)=1-(r+1)x^r$, 
so the lower slope bound is $-r$, yielding an admissible 
range for $a$ that changes with $r$.
\end{remark}

\section{Conclusion}\label{sec:conclusion}

Sarmanov copulas are an attractive family of copulas because 
they remain analytically tractable, accommodate asymmetry 
through distinct kernels in each margin, and often lead to 
closed-form expressions for standard measures of association. 
Their practical use, however, is hindered by admissibility: 
even in the bivariate case, validity requires a nonnegativity 
constraint on the copula density, while in dimension $d$, 
the $d$-increasing property becomes combinatorial 
and quickly unmanageable for most kernels. 
The main contribution of this paper is to reformulate
this admissibility problem as a stochastic construction problem. 
By representing bivariate Sarmanov copulas as Bernoulli-indexed mixtures 
of simple one-dimensional kernels, we obtain copulas 
that are valid by construction, with feasibility encoded 
throughby the existence of an underlying Bernoulli probability mass function. 
The dependence parameter is carried entirely by a latent Bernoulli pair, 
while each margin is built from two cdfs whose mixture leads to a uniform distribution. 
This dependence parameter is proportional to the covariance of the Bernoulli pair, 
and the admissible range is the Fréchet--Hoeffding bounds for that covariance. 
This provides interpretive insight into the nature of dependence 
in Sarmanov copulas and a simple criterion for assessing admissibility. 
This also yields a useful framework to build and simulate 
asymmetric Sarmanov copulas and recovers, within a common framework, 
several classical Huang--Kotz and Bairamov--Kotz--Bekçi 
specifications together with their admissible parameter ranges.

The same ideas extend naturally to higher dimensions, 
and we show that the exact stochastic representation in $d$ dimensions 
yields a copula with the familiar FGM subset expansion, 
where all interaction coefficients are normalized centred mixed moments 
of the latent Bernoulli random vector. 
This distinguishes marginal design (via univariate kernel factors) 
from dependence design (via the Bernoulli distribution) and bypasses 
direct verification of the $d$-increasing property. 
In three dimensions, the resulting formulation separates the pairwise
pairwise from three-way interactions and leads to closed-form 
expressions for orthant-based extensions of Spearman's rho.

Finally, we show that  ``powered'' Sarmanov models 
of the form $u_1u_2(1+ah_1(u_1)h_2(u_2))^r$, where $r$ is an integer, 
have a stochastic representation in terms of block maxima of 
unpowered Sarmanov copulas.

Several extensions are natural next steps for this work. 
Methodologically, the latent-Bernoulli representation suggests 
estimation and testing procedures that exploit conditional independence 
given the Bernoulli indices (e.g., likelihood-based or EM-type strategies) 
and motivates structured Bernoulli models (exchangeable, factor, sparse, 
or partially exchangeable) to control parameter growth in high dimensions. 
Theoretically, it would be valuable to extend the representation 
beyond single perturbations and to characterize which base constructions 
admit fractional block-maximum powers with max-infinite divisibility.
From an applied perspective, future work could use these copulas 
for risk aggregation, stress testing, and scenario generation 
in finance and insurance. 
Current applications focus mainly on bivariate models,
but the $d$-variate constructions here may facilitate higher-dimensional applications.

\section*{Acknowledgements}

The author acknowledges financial support from the Natural Sciences and Engineering Research Council of Canada (RGPIN-2025-06879).

\appendix

\section{Omitted proofs}\label{sec:proofs}

\begin{proof}[Proof of Theorem \ref{thm:stochastic}]
We first verify that the auxiliary functions 
$F_{1,[0]}, F_{1,[1]}, F_{2,[0]}, F_{2,[1]}$ are valid cdfs on $[0,1]$. 
We then show that the stochastic representation yields uniform marginals 
and the claimed Sarmanov copula. 
Note that $g_1(0)=g_1(1)=0$ implies
$F_{1,[i]}(0)=0$ and $F_{1,[i]}(1)=1$ for $i\in\{0,1\}$. 
Since $g_1$ is absolutely continuous,
$$F'_{1,[0]}(u)=1-\Lambda_1 \phi_1(u)\ge 0 \quad\text{a.e.},
\qquad F'_{1,[1]}(u)=1-\lambda_1 \phi_1(u)\ge 0 \quad\text{a.e.},$$
using $\phi_1\le 1/\Lambda_1$ a.e. and $\phi_1\ge 1/\lambda_1$ a.e. 
with $\lambda_1<0$ (so $\lambda_1 \phi_1\le 1$ a.e.). 
Hence $F_{1,[0]}$ and $F_{1,[1]}$ are nondecreasing and take values in $[0,1]$. 
The same argument applies to $F_{2,[0]}$ and $F_{2,[1]}$. 
We next verify that the stochastic representation leads to uniform marginals. 
Conditioning on $I_1$ gives
$$\p(U_1\le u)=(1-\pi_1)F_{1,[0]}(u)+\pi_1 F_{1,[1]}(u) =u-\left(\Lambda_1(1-\pi_1)+\lambda_1\pi_1\right)g_1(u).$$
Since $\pi_1=\Lambda_1/(\Lambda_1 - \lambda_1)$, 
we have $(1-\pi_1)\Lambda_1+\pi_1\lambda_1=0$, 
hence $\p(U_1\le u)=u$ for all $u\in[0,1]$ and $U_1\sim\mathrm{Unif}(0,1)$. 
The same holds for $U_2\sim\mathrm{Unif}(0,1)$. 
Finally, for $u_1,u_2\in[0,1]$, conditioning on $(I_1,I_2)$ 
and using conditional independence of $U$ and $V$ given $(I_1, I_2)$, 
the copula is
$$C(u_1, u_2)=\p(U_1\le u_1,U_2\le u_2)=
   \E\left[F_{1,[I_1]}(u_1)F_{2,[I_2]}(u_2)\right],$$
where
$$F_{1,[I_1]}(u)=u-g_1(u)\left((1-I_1)\Lambda_1+I_1\lambda_1\right),
\qquad F_{2,[I_2]}(u)=u-g_2(u)\left((1-I_2)\Lambda_2+I_2\lambda_2\right).$$
Expanding the product and regrouping terms yields
$$C(u_1,u_2)=u_1 u_2-u_1 g_2(u_2)\E\left[(1-I_2)\Lambda_2+I_2\lambda_2\right] -u_2 g_1(u_1)\E\left[(1-I_1)\Lambda_1+I_1\lambda_1\right] +g_1(u_1)g_2(u_2) \Xi,$$
with $\Xi:=\E\left[\left((1-I_1)\Lambda_1 + I_1\lambda_1\right) \left((1-I_2)\Lambda_2 + I_2\lambda_2\right)\right]$. 
The two linear terms vanish since 
$$\E\left[(1 - I_1) \Lambda_1 + I_1 \lambda_1\right]=
   (1 - \pi_1) \Lambda_1 + \pi_1 \lambda_1=0$$
and likewise for $I_2$. 
We also have
$$(1-I_1)\Lambda_1+I_1\lambda_1=
\left(\lambda_1-\Lambda_1\right)(I_1-\pi_1),\qquad
(1-I_2)\Lambda_2+I_2\lambda_2=\left(\lambda_2-\Lambda_2\right)(I_2-\pi_2).$$
Using $\left(\lambda_k-\Lambda_k\right)\pi_k=-\Lambda_k$ for $k=1,2$, 
we obtain $\Xi=\Lambda_1\Lambda_2\theta$. 
Therefore,
$$C(u_1,u_2)=u_1 u_2+\Lambda_1\Lambda_2\theta g_1(u_1)g_2(u_2),$$
which is the desired Sarmanov form with $a=\Lambda_1\Lambda_2\theta$.   
\end{proof}

\begin{proof}[Proof of Theorem \ref{thm:global-extrema-rho}]
Recall that $\rho_S(C)=12aG_1G_2$ with $G_k:=\int_0^1 g_k(u)\d u$. 
We bound $|G_1|$; the same argument applies to $|G_2|$. 
From $g_1(0)=0$ and $\phi_1\le 1/\Lambda_1$ a.e., 
$g_1(u)=\int_0^u\phi_1\le u/\Lambda_1$. 
From $g_1(1)=0$ and $\phi_1\ge 1/\lambda_1$ a.e., 
$g_1(u)=-\int_u^1\phi_1\le -(1/\lambda_1)(1-u)=|1/\lambda_1|(1-u)$. 
These envelopes intersect at $u_0:=\Lambda_1/(\Lambda_1+|\lambda_1|)$, hence
$$G_1\le \int_0^{u_0}\frac{u}{\Lambda_1}\d u+\int_{u_0}^1\left| \frac{1}{\lambda_1} \right|(1-u)\d u =\frac{1}{2(\Lambda_1+|\lambda_1|)}.$$
Applying the same bound to $-g_1$ gives $|G_1|\le 1/(2(\Lambda_1+|\lambda_1|))$. 
Similarly,
$$|G_2|\le \frac{1}{2(\Lambda_2+|\lambda_2|)}.$$

We now find the admissible range of $a$. 
The density is $c(u_1,u_2)=1+a\phi_1(u_1)\phi_2(u_2)\ge 0$ a.e. 
Since $\phi_k\in[1/\lambda_k,1/\Lambda_k]$ a.e., 
the extreme values of $\phi_1\phi_2$ occur at the corners. 
Thus, if $a\ge 0$ then $1+a/(\Lambda_1\lambda_2)\ge 0$ and 
$1+a/(\lambda_1\Lambda_2)\ge 0$, so
$a\le \min\left(\Lambda_1|\lambda_2|,|\lambda_1|\Lambda_2\right)$. 
If $a\le 0$, writing $a=-|a|$, then $1-|a|/(\Lambda_1\Lambda_2)\ge 0$ 
and $1-|a|/(\lambda_1\lambda_2)\ge 0$, 
so $|a|\le \min\left(\Lambda_1\Lambda_2,|\lambda_1\lambda_2|\right).$

Using $|\rho_S|=12|a||G_1||G_2|$ and the bounds on $|G_1|,|G_2|$, 
we get
$$|\rho_S| \le \frac{3|a|}{(\Lambda_1+|\lambda_1|)(\Lambda_2+|\lambda_2|)}.$$
Set $x=|\lambda_1|$, $y=\Lambda_1$, $s=|\lambda_2|$, 
and $t=\Lambda_2$ so that $x, y, s$ and $t$ are positive. 
If $a\ge 0$, then $|a|\le \min(ys,xt)$ and 
$$|\rho_S|\le \frac{3\min(ys,xt)}{(x+y)(s+t)}.$$
Since $(x+y)(s+t)=xs+yt+xt+ys\ge xt+ys+2\sqrt{xsyt}=(\sqrt{xt}+\sqrt{ys})^2\ge 4\min(xt,ys)$, we obtain $|\rho_S|\le 3/4$. 
If $a\le 0$, then $|a|\le \min(yt,xs)$ and hence
$$|\rho_S|\le \frac{3\min(xs,yt)}{(x+y)(s+t)}.$$
Since $(x+y)(s+t)=xs+yt+xt+ys\ge xs+yt+2\sqrt{xtys}=
(\sqrt{xs}+\sqrt{yt})^2\ge 4\min(xs,yt)$, we again get $|\rho_S|\le 3/4$.

Finally, the bound is sharp. 
Let $g_1^*(u)=g_2^*(u)=\min(u,1-u)$ and 
$\phi_k^*(u)=\id{u\in[0,1/2)}-\id{u\in(1/2,1]}$ a.e. 
Then $\Lambda_1=\Lambda_2=1$, $\lambda_1=\lambda_2=-1$, 
and $G_1^*=G_2^*=\int_0^1\min(u,1-u)\,\d u=1/4$. 
For $a=\pm 1$, the copula $C_\pm(u_1,u_2)=u_1u_2\pm g_1^*(u_1)g_2^*(u_2)$ 
has density $1\pm \phi_1^*(u_1)\phi_2^*(u_2)\in\{0,2\}$ a.e., 
hence is valid, and $\rho_S(C_\pm)=12(\pm 1)(1/4)^2=\pm 3/4$.
\end{proof}

\section{Example kernels}\label{sec:examples}

We collect in Table \ref{tab:kernels-kappa-Lambda-lambda} several known 
Sarmanov kernels from the literature, along with 
their corresponding values of $\kappa=\int_0^1 g(u) \d u$, 
the parameters $\Lambda$ and $\lambda$, and references. 
We also introduce several additional kernels that satisfy 
the regularity conditions of Theorem \ref{thm:global-extrema-rho}. 
The table uses the notation $B(\cdot,\cdot)$ as the Beta function 
and $\Phi$ as the standard normal cdf. 
Some of these kernels follow the strategy of \cite{fischer2007constructing}, 
who designed kernels by taking the copula generator as 
$u\mapsto f(F^{-1}(u))$ associated with a suitable univariate distribution.

\newcolumntype{K}{>{$\displaystyle}X<{$}}   
\newcolumntype{M}{>{$\displaystyle}c<{$}}   

\newcommand{\gHS}{\ensuremath{\sin(\pi u)/\pi}}
\newcommand{\gCau}{\ensuremath{\sin^{2}(\pi u)/\pi}}
\newcommand{\gLap}{\ensuremath{\begin{cases}
u,&0<u\le \frac12,\\
1-u,&\frac12<u\le 1,
\end{cases}}}
\newcommand{\gEighteen}{\ensuremath{\begin{cases}
1,&0<u\le \frac14,\\
\frac{1/4}{u},&\frac14<u\le \frac34,\\
\frac{1-u}{u},&\frac34<u\le 1,
\end{cases}}}
\newcommand{\gNineteen}{\ensuremath{\begin{cases}
2,&0<u\le \frac13,\\
\frac{1-u}{u},&\frac13<u\le 1,
\end{cases}}}
\newcommand{\gTwenty}{\ensuremath{\begin{cases}
1,&0<u\le \frac23,\\
\frac{2(1-u)}{u},&\frac23<u\le 1,
\end{cases}}}
\newcommand{\gTwentyOne}{\ensuremath{\begin{cases}
3u,&0<u\le \frac14,\\
1-u,&\frac14<u\le 1,
\end{cases}}}

\begin{table}[ht]
\centering
\footnotesize
\setlength{\tabcolsep}{5pt}
\renewcommand{\arraystretch}{2.2}
\addtolength{\aboverulesep}{1pt}
\addtolength{\belowrulesep}{1pt}

\begin{tabularx}{\linewidth}{@{}>{\hsize=0.15\hsize}r >{\hsize=.25\hsize}K M M M >{\hsize=1\hsize}>{\raggedright\arraybackslash}X@{}} \toprule
\# & g(u)                                       & \kappa=\int_0^1 g(u)\d u          & \Lambda & \lambda                                 & \text{Reference} \\ \midrule
1  & u(1-u)                                     & 1/6                               & 1       & -1                                      & \cite{morgenstern1956einfache}, \cite{gumbel1960bivariate} \\
2  & u(1-u^{p})\ \ (p>0)                        & \frac{p}{2(p+2)}                  & 1       & -1/p                                    & \cite{huang1999modifications} \\
3  & u(1-u)^{q}\ \ (q>1)                        & \frac{1}{(q+1)(q+2)}              & 1       & -\left(\frac{q+1}{q-1}\right)^{q-1}     & \cite{huang1999modifications} \\
4  & u(1-u^{p})^{q}\ \ (p>0,q>1)                & \frac1p B\left(\frac{2}{p},q+1\right)         & 1       & -\frac{(1+pq)^{q-1}}{p^{q}(q-1)^{q-1}}  & \cite{bairamov2001new} \\
5  & \gHS                                       & \frac{2}{\pi^{2}}                 & 1       & -1                                      & \cite{fischer2007constructing} \\
6  & \gCau                                      & \frac{1}{2\pi}                    & 1       & -1                                      & \cite{fischer2007constructing} \\
7  & \gLap                                      & 1/4                               & 1       & -1                                      & \cite{fischer2007constructing} \\
8  & u^{b}(1-u)^{a}\ \ (a>1,\ b>1)              & B(b+1,a+1)                        & \eqref{eq:lai-xie-Lambda} &  \eqref{eq:lai-xie-lambda}     & \cite{lai2000new} \\
9  & \dfrac{u(u-1)}{2}                          & -\dfrac{1}{12}                    & 2       & -2                                      & \cite{tinglee1996properties} \\
10 & \dfrac{u^{k+1}-u}{k+1}, \,k\in\mathbb{N}   & -\dfrac{k}{2(k+1)(k+2)}           & \dfrac{k+1}{k}  & -(k+1)                          & \cite{tinglee1996properties} \\
11 & (1-e^{-u})-(1-e^{-1})u                     & \dfrac{3-e}{2e}                   & e       & -\dfrac{e}{e-2}                         & \cite{tinglee1996properties} \\
12 &  \eqref{eq:norm-lee}  & \frac{e^{2/3}}{\sqrt{3}}\left(\Phi\left(\frac{1}{\sqrt{3}}\right)-\frac{1}{2}\right)   & \frac{1}{e-\frac{e^{2/3}}{\sqrt{3}}}    & -\sqrt{3}e^{-2/3}    & \cite{tinglee1996properties} \\
13 & u\frac{e^{2(1-u)}-1}{e^{2}-1}              & \frac{e^{2}-5}{4(e^{2}-1)}        & 1       & -\frac{e^{2}-1}{2}                      & --- \\
14 & u\frac{1-u}{1+u}                           & \frac{3}{2}-\log 4                & 1       & -2                                      & --- \\
15 & u(1-u)(1-2u)                               & 0                                 & 1       & -2                                      & --- \\
16 & (1-u)\sin(\pi u)/\pi                       & \frac{1}{\pi^{2}}                 & 1       & \approx -2.25 \ \ \eqref{eq:sin}                         & --- \\
17 & u\frac{1-u}{1+u^{2}}                       & -1+\frac{\log 2}{2}+\frac{\pi}{4} & 1      & -2                                      & --- \\
18 & u\frac{(1-u)^{2}}{1+u}                      & \frac{17}{6}-\log 16              & 1       & \frac{1}{3\sqrt[3]{4}-5}                & --- \\
19 & u(1-u)e^{-u}                               & -1+\frac{3}{e}                    & 1       & -e                                      & --- \\
20 & \gTwentyOne                                & 3/8                               & 1/3     & -1                                      & --- \\
\bottomrule
\end{tabularx}

\caption{Table of selected kernels $g(u)$ with their associated $\kappa$, $\Lambda$, and $\lambda$ values.}\label{tab:kernels-kappa-Lambda-lambda}
\end{table}
\FloatBarrier

The following expressions are used in Table \ref{tab:kernels-kappa-Lambda-lambda}:
\begin{equation}\label{eq:lai-xie-Lambda}
\Lambda = \frac{(a+b)^{a+b-2}\sqrt{a+b-1}} {\sqrt{ab} \left(b-\frac{\sqrt{ab}}{\sqrt{a+b-1}}\right)^{b-1}\left(a+\frac{\sqrt{ab}}{\sqrt{a+b-1}}\right)^{a-1}};
\end{equation}
\begin{equation}\label{eq:lai-xie-lambda}
\lambda = -\frac{(a+b)^{a+b-2}\sqrt{a+b-1}} {\sqrt{ab} \left(b+\frac{\sqrt{ab}}{\sqrt{a+b-1}}\right)^{b-1}\left(a-\frac{\sqrt{ab}}{\sqrt{a+b-1}}\right)^{a-1}};
\end{equation}
\begin{equation}\label{eq:norm-lee}
   \frac{e^{2/3}}{\sqrt{3}}\left(\Phi\left(\sqrt{3}\Phi^{-1}(u)+\frac{2}{\sqrt{3}}\right)-u\right);
\end{equation}
\begin{equation}\label{eq:sin}
   \lambda =- \pi \frac{\sqrt{y^{2}+4}}{y^{2}+2},
\end{equation}
where $y\in(0,\pi/2)$ is the solution for $y$ in $y\tan y=2$. Numerically, \eqref{eq:sin} gives $\lambda \approx -2.2585010314.$

One particularly interesting kernel to highlight is kernel \#15 
in Table \ref{tab:kernels-kappa-Lambda-lambda}, 
defined by $g(u)=u(1-u)(1-2u)$. 
This is also the kernel $\Phi_2$ from \cite{hamadou2025extension}, 
who motivate this kernel as the primitive of the second 
normalized shifted Legendre polynomial, 
providing an orthogonal extension of the classical FGM term.

This kernel has zero signed area, that is, $\kappa=0$, 
and therefore induces a Sarmanov copula with 
Spearman's $\rho_S$ and Kendall's $\tau$ equal to zero for all admissible $a$. 
Nevertheless, the copula is not the independence copula unless $a=0$. 
The resulting copula is 
$$C_a(u_1,u_2)=u_1u_2 +au_1(1-u_1)(1-2u_1) u_2(1-u_2)(1-2u_2), 
\qquad (u_1,u_2)\in[0,1]^2$$
for $a\in [-1,2].$ 
Also note that $g(u)$ changes sign on $(0,1)$, 
being positive on $(0,1/2)$ and negative on $(1/2,1)$. 
This implies that Theorem \ref{thm:sm-order-fixed-kernels} 
does not hold for this kernel, since the associated distributions 
$F_{1,[0]}$ and $F_{1,[1]}$ are not comparable in the stochastic order.

\section{Comparison of Farlie and Sarmanov copulas}\label{app:comparison-three-classes}

We compare three classical bivariate dependence constructions 
that are often presented under different names. 
The first is the Farlie-type copula specification from \cite{farlie1960performance}:
$$C_\alpha(u_1,u_2)=u_1u_2(1+\alpha h_1(u_1)h_2(u_2)),\qquad (u_1,u_2)\in[0,1]^2,$$
parametrized by a scalar $\alpha$ and fixed kernels $h_1, h_2$. 
The second is the additive copula specification, 
proposed by \cite{rodriguez-lallena2004new}, defined as
$$C(u_1,u_2)=u_1u_2+g_1(u_1)g_2(u_2),\qquad (u_1,u_2)\in[0,1]^2,$$
parametrized by functions $g_1,g_2$ subject to copula admissibility conditions. 
The third is the Sarmanov construction 
\cite{sarmanov1966generalized, tinglee1996properties}, 
defined by its density
$$f(x_1,x_2)=f_1(x_1)f_2(x_2)(1+\alpha_{12}\phi_1(x_1)\phi_2(x_2)),$$
with bounded kernels $\phi_1,\phi_2$ satisfying mean-zero 
and positivity constraints. 
When the marginals are continuous, this induces a copula through 
the probability integral transform.

These three constructions are closely related. 
Under absolute continuity, they all lead to copula densities of the form
$$c(u_1,u_2)=1+\gamma \psi(u_1)\eta(u_2),$$
for a scalar $\gamma$ and suitable univariate factors $\psi,\eta$. 
The differences between the families are therefore not 
in the density-level factorization itself, 
but in (i) the level at which the factorization is imposed 
(copula level versus joint density level), and 
(ii) the boundary and regularity restrictions required of the factor functions.

For the Farlie copulas, assume $h_1$ and $h_2$ are absolutely continuous and 
that $g_1(u)=uh_1(u)$ and $g_2(u)=uh_2(u)$ have essentially bounded derivatives. 
Then the copula is absolutely continuous with density
$$c_\alpha(u_1,u_2)=1+\alpha \phi_{1}(u_1)\phi_{2}(u_2),\qquad \phi_{1}(u)=(uh_1(u))', \quad \phi_{2}(u)=(uh_2(u))'.$$
Hence, the density-level factors must be derivatives 
of $u_1 h_1(u_1)$ and $u_2 h_2(u_2)$.

For the additive specification, if $g_1$ and $g_2$ are 
absolutely continuous, then $C$ is absolutely continuous and
$$c(u_1,u_2)=1+g_1'(u_1)g_2'(u_2)\qquad\text{for a.e. }(u_1,u_2)\in(0,1)^2.$$
Every Farlie-type copula is of the additive form through the identification
$$g_1(u)= \alpha uh_1(u),\qquad g_2(u)=uh_2(u).$$
Thus, the multiplicative Farlie parametrization 
is a special case of the additive parametrization.

A partial converse holds, but it depends on boundary regularity at the origin. 
Fix $\alpha\neq 0$ and consider $C(u_1,u_2)=u_1u_2+g_1(u_1)g_2(u_2)$. 
For $u_1,u_2>0$ one may rewrite $u_1u_2+g_1(u_1)g_2(u_2)$ as 
$u_1u_2(1+\alpha h_1(u_1)h_2(u_2))$, with
$h_1(u_1)= g_1(u_1)/(\alpha u_1)$ and $h_2(u_2)=g_2(u_2)/u_2$. 
To obtain kernels on $[0,1]$ compatible with the Farlie-type framework, 
the possible indeterminate forms at $u=0$ and $v=0$ must be excluded. 
One can, if the limits exist, define
$$h_1(0):=\lim_{u\downarrow 0}\frac{g_1(u)}{u\alpha},\qquad 
h_2(0):=\lim_{u\downarrow 0}\frac{g_2(u)}{u},$$
if the resulting extensions are bounded and satisfy 
the regularity conditions imposed on Farlie kernels. 
When these limits fail to exist or when division by $u$ 
destroys absolute continuity, an additive copula may remain admissible 
but fall outside the Farlie-type class as defined here, 
but the stochastic representation of Theorem \ref{thm:stochastic} still applies. 
Indeed, we have put conditions on $uh_1(u)$ and $uh_2(u)$ 
rather than on $h_1$ and $h_2$ directly, which allows for 
more general boundary behaviour at the origin, 
and coincides with the additive class of \cite{rodriguez-lallena2004new}.

For the Sarmanov construction, assume continuous marginals 
and define $U_1=F_1(X_1)$ and $U_2=F_2(X_2)$. 
Then the induced copula density is
$$c(u_1,u_2)=1+\alpha\tilde\phi(u_1)\tilde\psi(u_2),\qquad 
\tilde\phi(u)=\phi_1(F_1^{-1}(u)),\quad \tilde\psi(v)=\phi_2(F_2^{-1}(v)),$$
with $\int_0^1\tilde\phi(u)\d u=\int_0^1\tilde\psi(v)\d v=0$ 
by the mean-zero constraints. 
Integrating yields the additive copula representation
$$C(u_1,u_2)=u_1u_2+\alpha\Phi(u_1)\Psi(u_2),\qquad 
\Phi(u)=\int_0^u \tilde\phi(s)\d s,\quad \Psi(u)=\int_0^u \tilde\psi(t)\d t,$$
so, under absolute continuity, Sarmanov copulas coincide 
with the additive class of \cite{rodriguez-lallena2004new}, 
up to an equivalent reparameterization.

Therefore, under absolute continuity, 
the three families share the exact factorization as the Sarmanov copulas, 
and Theorem \ref{thm:stochastic} applies to all three constructions, 
provided the boundary and regularity conditions are satisfied. 
The main differences lie in the interpretation of the factor functions, 
the level at which the factorization is imposed, 
and the boundary conditions required for admissibility.

\end{document}